\documentclass[11pt]{amsart}
\usepackage{amsopn}
\usepackage{amssymb, amscd}
\usepackage{multirow}
\usepackage{graphicx, graphics, epsfig}
\usepackage{faktor}
\usepackage{caption}
\usepackage{booktabs}
\usepackage{ textcomp }
\usepackage{hyperref}
\usepackage{xcolor}

\newcommand\restr[2]{{% we make the whole thing an ordinary symbol
		\left.\kern-\nulldelimiterspace % automatically resize the bar with \right
		#1 % the function
		\vphantom{\big|} % pretend it's a little taller at normal size
		\right|_{#2} % this is the delimiter
}}

\newcommand{\nc}{\newcommand}

\newcommand{\suchthat}{\;\ifnum\currentgrouptype=16 \middle\fi|\;}
\nc{\zg}{\mathfrak{z} } \nc{\ngo}{\mathfrak{n} } \nc{\kg}{\mathfrak{k} }  
\nc{\mg}{\mathfrak{m} } \nc{\bg}{\mathfrak{b} } \nc{\ggo}{\mathfrak{g} } \nc{\hgo}{\mathfrak{h} }
\nc{\ggob}{\overline{\mathfrak{g}} } \nc{\sog}{\mathfrak{so} }
\nc{\sug}{\mathfrak{su} } \nc{\spg}{\mathfrak{sp} } \nc{\slg}{\mathfrak{sl} }
\nc{\glg}{\mathfrak{gl} } \nc{\cg}{\mathfrak{c} } \nc{\rg}{\mathfrak{r} }
\nc{\hg}{\mathfrak{h} } \nc{\tg}{\mathfrak{t} } \nc{\ug}{\mathfrak{u} }
\nc{\dg}{\mathfrak{d} } \nc{\ag}{\mathfrak{a} } \nc{\pg}{\mathfrak{p} }
\nc{\sg}{\mathfrak{s} } \nc{\affg}{\mathfrak{aff} } \nc{\qg}{\mathfrak{q} }

\nc{\pca}{\mathcal{P}} \nc{\nca}{\mathcal{N}} \nc{\lca}{\mathcal{L}}
\nc{\oca}{\mathcal{O}} \nc{\mca}{\mathcal{M}} \nc{\tca}{\mathcal{T}}
\nc{\aca}{\mathcal{A}} \nc{\cca}{\mathcal{C}} \nc{\gca}{\mathcal{G}}
\nc{\sca}{\mathcal{S}} \nc{\hca}{\mathcal{H}} \nc{\bca}{\mathcal{B}}
\nc{\dca}{\mathcal{D}} \nc{\val}{\operatorname{val}}
\nc{\Spec}{Spec}

\nc{\vp}{\varphi} \nc{\ddt}{\frac{d}{dt}} \nc{\dds}{\frac{d}{ds}}
\nc{\dpar}{\frac{\partial}{\partial t}} \nc{\im}{\mathrm{i}}

\nc{\SO}{\mathrm{SO}} \nc{\Spe}{\mathrm{Sp}} \nc{\Sl}{\mathrm{SL}}
\nc{\SU}{\mathrm{SU}} \nc{\Or}{\mathrm{O}} \nc{\U}{\mathrm{U}} \nc{\Gl}{\mathrm{GL}}
\nc{\Se}{\mathrm{S}} \nc{\Cl}{\mathrm{Cl}} \nc{\Spein}{\mathrm{Spin}}
\nc{\Pin}{\mathrm{Pin}} \nc{\G}{\mathrm{GL}_n(\RR)} \nc{\g}{\mathfrak{gl}_n(\RR)}

\nc{\RR}{{\Bbb R}} \nc{\HH}{{\Bbb H}} \nc{\CC}{{\Bbb C}} \nc{\ZZ}{{\Bbb Z}}
\nc{\FF}{{\Bbb F}} \nc{\NN}{{\Bbb N}} \nc{\QQ}{{\Bbb Q}} \nc{\PP}{{\Bbb P}} \nc{\OO}{{\Bbb O}}

\nc{\vs}{\vspace{.2cm}} \nc{\vsp}{\vspace{1cm}} \nc{\ip}{\langle\cdot,\cdot\rangle}
\nc{\ipp}{(\cdot,\cdot)} \nc{\la}{\langle} \nc{\ra}{\rangle} \nc{\unm}{\tfrac{1}{2}}
\nc{\unc}{\tfrac{1}{4}} \nc{\und}{\tfrac{1}{16}} \nc{\no}{\vs\noindent}
\nc{\lam}{\Lambda^2(\RR^n)^*\otimes\RR^n} \nc{\tangz}{{\rm T}^{\rm Zar}}
\nc{\nor}{{\sf n}}  \nc{\mum}{/\!\!/} \nc{\kir}{/\!\!/\!\!/}
\nc{\Ri}{\tfrac{4\Ric_{\mu}}{||\mu||^2}} \nc{\ds}{\displaystyle}
\nc{\ben}{\begin{enumerate}} \nc{\een}{\end{enumerate}} \nc{\f}{\frac}
\nc{\lb}{[\cdot,\cdot]} \nc{\isn}{\tfrac{1}{||v||^2}}
\nc{\gkp}{(\ggo=\kg\oplus\pg,\ip)} \nc{\ukh}{(\ug=\kg\oplus\hg,\ip)}
\nc{\tgkp}{(\tilde{\ggo}=\kg\oplus\pg,\ip)}
\nc{\wt}{\widetilde} \nc{\mm}{M}
\nc{\iop}{\mathtt{i}} \nc{\jop}{\mathtt{j}}

\nc{\Hess}{\operatorname{Hess}} \nc{\ad}{\operatorname{ad}}
\nc{\Ad}{\operatorname{Ad}} \nc{\rank}{\operatorname{rank}}
\nc{\Irr}{\operatorname{Irr}} \nc{\End}{\operatorname{End}}
\nc{\Aut}{\operatorname{Aut}} \nc{\Inn}{\operatorname{Inn}}
\nc{\Aff}{\operatorname{Aff}} \nc{\aff}{\operatorname{aff}}
\nc{\Der}{\operatorname{Der}} \nc{\Ker}{\operatorname{Ker}}
\nc{\Iso}{\operatorname{Iso}} \nc{\Diff}{\operatorname{Diff}}
\nc{\Lie}{\operatorname{L}} \nc{\tr}{\operatorname{tr}} \nc{\dif}{\operatorname{d}}
\nc{\sen}{\operatorname{sen}} \nc{\modu}{\operatorname{mod}}
\nc{\CRic}{\operatorname{PP}} \nc{\Cric}{\operatorname{P}} \nc{\Ricci}{\operatorname{Ric}}
\nc{\sym}{\operatorname{sym}} \nc{\herm}{\operatorname{herm}} \nc{\symac}{\operatorname{sym^{ac}}}
\nc{\symc}{\operatorname{sym^{c}}} \nc{\scalar}{\operatorname{sc}}
\nc{\grad}{\operatorname{grad}} \nc{\ricci}{\operatorname{Rc}}
\nc{\Nor}{\operatorname{Norm}}  \nc{\ricc}{\operatorname{Rc^{c}}}
\nc{\Ricc}{\operatorname{Ric^{c}}} \nc{\ricac}{\operatorname{Rc^{ac}}}
\nc{\Ricac}{\operatorname{Ric^{ac}}} \nc{\Riem}{\operatorname{Rm}}
\nc{\riccig}{\operatorname{ric^{\gamma}}} \nc{\Rin}{\operatorname{M}}
\nc{\Le}{\operatorname{L}} \nc{\tang}{\operatorname{T}}
\nc{\level}{\operatorname{level}} \nc{\rad}{\operatorname{r}}
\nc{\abel}{\operatorname{ab}} \nc{\CH}{\operatorname{CH}} \nc{\Cone}{\operatorname{C}} \nc{\CCone}{\operatorname{CC}}
\nc{\mcc}{\operatorname{mcc}} \nc{\Adj}{\operatorname{Adj}}
\nc{\Order}{\operatorname{O}}  \nc{\inj}{\operatorname{inj}} \nc{\proy}{\operatorname{pr}}
\nc{\vol}{\operatorname{vol}} \nc{\Diag}{\operatorname{Dg}} \nc{\Diagg}{\operatorname{Diag}}
\nc{\Ima}{\operatorname{Im}} \nc{\Rea}{\operatorname{Re}}
\nc{\spann}{\operatorname{span}}
\nc{\id}{\operatorname{id}}

\theoremstyle{plain}
\newtheorem{theorem}{Theorem}[section]
\newtheorem{proposition}[theorem]{Proposition}
\newtheorem{corollary}[theorem]{Corollary}
\newtheorem{lemma}[theorem]{Lemma}
\newtheorem{question}{Question}
\newtheorem{mainthm}{Theorem}

\theoremstyle{definition}
\newtheorem{definition}[theorem]{Definition}

\theoremstyle{remark}
\newtheorem{remark}[theorem]{Remark}

\newtheorem{example}[theorem]{Example}

\title[On post-Lie algebras structures coming from \\ simply transitive NIL-affine actions]{On post-Lie algebras structures coming from \\ simply transitive NIL-affine actions}

\author{Jonas Der\'e} 
\address{KU Leuven Kulak, E. Sabbelaan 53, BE-8500 Kortrijk, Belgium}
\email{jonas.dere@kuleuven.be}

\author{Marcos Origlia}
\address{Universidad Nacional de C\'ordoba, 5000 C\'ordoba, Argentina}
\email{marcos.origlia@unc.edu.ar}

\thanks{The first author is supported by a start-up grant of KU Leuven (grant number 3E210985). The second author is supported by CONICET (PIBAA grant number 28720210100001CO),
ANPCyT and SECyT-UNC (grant number PICT-2021-0628) and FWO (scientific stay number V510423N)}

\begin{document}

\maketitle

\begin{abstract}
Given a simply connected solvable Lie group $G$, there always exists NIL-affine action $\rho: G \to \Aff(H)$ on a nilpotent Lie group $H$ such that $G$ acts simply transitively. The question whether this is always possible for $H = \RR^n$ abelian was known as Milnor's question, with a negative answer due to a counterexample of Benoist. This counterexample is based on a correspondence between certain affine actions $\rho: G \to \Aff(\RR^n)$ and left-symmetric structures on the corresponding Lie algebra $\ggo$ of $G$, where simply transitive actions correspond exactly to the so-called complete left-symmetric structures. In general however, the question remains open which solvable Lie groups $G$ can act on which nilpotent Lie groups $H$.

A natural candidate for a correspondence on the Lie algebra level is the notion of post-Lie algebra structures, which form the natural generalization of left-symmetric structures. In this paper, we show that every simply transitive NIL-affine action of $G$ on a nilpotent Lie group $H$ indeed induces a post-Lie algebra structure on the pair of Lie algebras $(\ggo,\hgo)$. Moreover, we discuss a new notion of completeness for these structures in the case that $\hgo$ is $2$-step nilpotent, equivalent but different from the known definition for $H = \RR^n$. We then show that simply transitive actions exactly correspond to complete post-Lie algebra structures in the $2$-step nilpotent case. However, the questions how to define completeness in higher nilpotency classes remains open, as we illustrate with an example in the $3$-step nilpotent case.
\end{abstract}

\tableofcontents

\section{Introduction}

Let $G$ be a connected and simply connected (hereinafter abbreviated as $1$-connected) solvable Lie group of dimension $n$, then Milnor showed that $G$ admits a smooth action on $\RR^m$ with $m \geq n$ via affine transformations $\Aff(\RR^m)$ that is free, see \cite{miln77-1}. This raised the problem whether this result could be improved to a smooth action $\rho:G\to \Aff(\RR^n)$ such that $G$ acts simply transitive on $\RR^n$, meaning that there exists for every $x, y \in \RR^n$ a unique element $g \in G$ such that $\rho_g (x) = y$. This problem was known in the literature as \textbf{Milnor's question}, with some partial positive results, for instance in the case of nilpotent Lie groups having an expanding automorphism in \cite{sche72} or groups up to nilpotency class $3$ in \cite{sche74}.

A important step in understanding these actions was given in \cite{kim86-1}, where it was shown that simply transitive affine actions of $G$ are in one-to-one correspondence to \textbf{complete left-symmetric structures} (sometimes also called pre-Lie agebra structures) on the corresponding Lie algebra $\ggo$ of $G$. We will recall the exact concept in Definition \ref{complete_LS}, together with some background on the relation to affine actions. For almost $20$ years it was believed that the answer to Milnor's question was positive, meaning that any solvable Lie algebra $\ggo$ admits a complete left-symmetric structure. However, in \cite{beno95-1} Benoist gave the first example of a $11$-dimensional nilpotent (and thus in particular solvable) $1$-connected Lie group that does not admit such a simply transitive affine action. Later, Burde generalized this to a family of examples in dimension $11$ in \cite{bg95-1} and also exhibited a new example in dimension $10$, moreover proving that any filiform nilpotent Lie algebra up to dimension 9 admits a left-symmetric structure \cite{burd96-1}. The question whether there exists a nilpotent Lie algebra of dimension $\leq 9$ not admitting a complete left-symmetric structure is still open in general.

Because of the negative answer to Milnor's question, people broadened the geometric context and studied more general affine transformations on $1$-connected nilpotent Lie groups $H$. Let $\Aff(H) =  H \rtimes \Aut(H)$ be the group of affine transformations of $H$, which acts on $H$ via  $$(m,\alpha)\cdot n = m \cdot \alpha(n)$$ with $ m, n \in H, \alpha \in \Aut(H)$. A smooth action $\rho: G \to \Aff(H)$ is called \textbf{NIL-affine action} of the group $G$ on the nilpotent group $H$. Note that if $H=\RR^n$ is abelian, we obtain the usual affine group $\Aff(\RR^n)$ and thus the regular notion of affine actions. In the NIL-affine case, Dekimpe \cite{deki98-1} proved that for any solvable Lie group $G$, there exist a nilpotent Lie group $H$ and a NIL-affine action $\rho:G\to \Aff(H)$ letting $G$ act simply transitive on $H$, giving a positive answer to Milnor's question in this more general setting. Note that the converse is also true, namely if a group acts simply transitively on a $1$-connected nilpotent Lie group, it must be solvable, by the observation that Lie groups homeomorphic to $\RR^n$ with a faithful linear representation are always solvable.

However, the question which solvable Lie group $G$ can act on which nilpotent Lie group $H$ remains open in general. As mentioned above, the case when $H = \RR^n$ is abelian was characterized in \cite{kim86-1} by the relation to complete left-symmetric structures and thus properties of the corresponding Lie algebra $\ggo$ of $G$. Similarly, Burde, Dekimpe and Deschamps in \cite{bdd09-1} analyzed the particular case when the Lie group $G$ is nilpotent, by relating it to properties of the Lie algebras $\ggo$ and $\hgo$ corresponding to $G$ and $H$. As an application, they showed that for every $n \leq 5$ and any $1$-connected nilpotent Lie groups $G$ and $H$ of dimension $n$, there exists a simply transitive NIL-affine action $\rho: G \to \Aff(H)$. They also exhibited an $6$-dimensional nilpotent Lie group $H$ such that $\RR^6$ cannot act simply transitive by NIL-affine actions on $H$. Very recently, the general case of simply transitive NIL-affine actions of a solvable Lie group $G$ on a nilpotent Lie group $H$ was considered in \cite{DerOri21}. Indeed, the existence of such an action is linked to certain embeddings of the semi-simple splitting into $\aff(\hgo)$ in \cite[Theorem 3.6]{DerOri21}, making it possible to describe the possibilities up to dimension $\leq 4$.

A post-Lie algebra structure (abbreviated as PLAS from now on) on a pair $(\ggo,\hgo)$ of Lie algebras defined on the same vector space $V$ is a product on $V$ which is related to both Lie brackets and generalizes the notion of left-symmetric structure, see Definition \ref{def_PLAS}. The concept of a PLAS was introduced in \cite{vall7} for studying partition posets, see also \cite{loda98-1} for the relation to Koszul operads. Afterwards, many other applications were studies, including isospectral flows and Yang-Baxter equations as in \cite{eflmmk15-1}. Using the aforementioned result of \cite{bdd09-1}, Burde, Dekimpe and Vercammen in \cite{bdv12-1} gave a correspondence between simply transitive NIL-affine actions $\rho: G \to \Aff(H)$ for $G$ nilpotent and a subfamily of PLAS on the pair of Lie algebras $(\ggo,\hgo)$, where $\ggo$ is isomorphic to the Lie algebra corresponding to $G$. 

More specifically, a PLAS is called \textbf{complete} if and only the left multiplication maps are nilpotent, and \cite{bdd09-1} shows that complete PLAS correspond to simply transitive actions. This is in contrast to the original definition of completeness in \cite{kim86-1} that uses the right multiplication map. However, Segal studied left-symmetric structures in a more algebraic way, improving the understanding of completeness in \cite{seg92-1}, and these results imply that the two notions of completeness above correspond, as we will also explain in Section \ref{sec:complete}. 

The question remains open whether there exist a correspondence between NIL-affine actions of a solvable Lie group $G$ on a nilpotent Lie group $H$, and a subfamily of post-Lie algebra structures, similarly as in the nilpotent case. The first main goal of this paper is to show that each simply transitive action indeed induces a PLAS.

\begin{mainthm}[Corollary \ref{cor:PLAS}]\label{thm:mainA}
	Every simply transitive action $\rho: G \to \Aff(H)$ with corresponding Lie algebras $\ggo$ and $\hgo$ induces a PLAS on a pair $(\tilde{\ggo},\hgo)$ with $\tilde{\ggo}$ isomorphic to $\ggo$.
\end{mainthm}
\noindent The proof in Section \ref{sec:toplas} below gives an explicit way to compute the PLAS from the derivative $\varphi = d\rho: \ggo \to \aff(\hgo)$. Secondly, we give a characterization of completeness in the $2$-step nilpotent case, generalizing both the notion of completeness from \cite{kim86-1} and \cite{bdd09-1}. 

\begin{mainthm}[Theorem \ref{thm main B}]\label{thm:mainB}
		Let	$G, H$ be $1$-connected Lie groups, with $G$ solvable and $H$ $2$-step nilpotent.  There exists a simply transitive NIL-affine action $\rho: G \to \Aff(H)$ if and only if there exists a Lie algebra $\tilde{\ggo}$, isomorphic to $\ggo$, defined under the same vector space as $\hgo$ and the pair $ (\tilde{\ggo},\hgo)$ admits a PLAS with $R_y - \frac{1}{2} \ad^{\hgo}_y$ nilpotent for every $y \in \hgo$, where $R_y$ is the right multiplication induced by the PLAS. 
\end{mainthm}

The term $- \frac{1}{2} \ad^{\hgo}_y$ did not occur in any of the previous definitions of completeness, nor is it sufficient in higher nilpotency class, see Example \ref{dim4_not enough}. We do show that it is equivalent to the known definitions, namely the case when $\hgo$ is abelian \cite{kim86-1} and the case with $\ggo$ nilpotent \cite{bdv12-1}.

\section{Preliminaries}

In this section we recall all known results about simply transitive affine and NIL-affine actions. The main goal is to characterize the pairs $(G,H)$ for $G$ and $H$ $1$-connected Lie groups with $G$ solvable and $H$ nilpotent such that there exists a simply transitive action $\rho: G \to \Aff(H)$, by relating it to the existence of post-Lie algebra structures.

In this paper, we will consider several Lie algebras on the same real vector space $V$. A Lie algebra is then defined by a Lie bracket $[\cdot,\cdot]: V \times V \to V$ where a subscript indicates which Lie algebra we are working with, for example the Lie algebra \(\ggo = (V, [\cdot,\cdot]_\ggo)\) has Lie bracket $[\cdot,\cdot]_\ggo$.

\subsection{Affine case $\RR^n$} In the classical case $H= \RR^n$, the existence of a simply transitive affine action of $G$ on $\RR^n$ was described in \cite{kim86-1}, by relating it to left-symmetric structures, of which we recall the definition. 

\begin{definition}\label{complete_LS}
	A \textit{left-symmetric structure} or LS on a Lie algebra $\ggo = (V,[\cdot,\cdot]_\ggo)$ is a bilinear product $$\cdot: V \times V \to V$$ satisfying, for all $x, y, z \in V$, the following properties:
	\begin{enumerate}
		\item $x \cdot y - y \cdot x= [x, y]_\ggo$
		\item $[x, y]_\ggo \cdot z =  x \cdot (y \cdot z) - y \cdot (x \cdot z)$.
	\end{enumerate}
	Moreover, a left-symmetric structure on $\ggo$ is called {\it complete} if for all $y \in V$ it holds that $\id + R_y$ is an isomorphism, where $R_y: V \to V$ is the right multiplication map. 
\end{definition}
Similarly, one can define a left-symmetric algebra as a vector space $V$ with the product $\cdot$ as in the definition above, and define the Lie algebra $\ggo$ via condition (1). As we will always start from a given Lie algebra $\ggo$, we do not use this concept in our paper.

Although completeness was originally defined as in the definition above, Segal gives a different (and a priori stronger) characterization in \cite[Theorem 1]{seg92-1}

\begin{theorem}
Let $\ggo$ be a Lie algebra, then a left-symmetric structure $\cdot$ is complete if and only if $R_y$ is nilpotent for every $y \in V$. 
\end{theorem}
\noindent The main method for this theorem was to show that completeness is preserved under field extensions. For nilpotent Lie algebras $\ggo$, completeness can also be characterized via the left multiplication map, see \cite[Theorem 2.1.]{kim86-1} and \cite[Theorem 2]{seg92-1}.

\begin{theorem}
	\label{thm_completenil}
Let $\ggo$ be a nilpotent Lie algebra, then a left-symmetric structure $\cdot$ is complete if and only if the left multiplication $L_x$ is nilpotent for every $x \in V$.
\end{theorem}

As there are different (equivalent) notions of completeness, it is unclear how to generalize it to the bigger class of post-Lie algebra structures that we will introduce in the next section. The importance of completeness lies in the relation to simply transitive affine actions. 

\begin{theorem}{\cite[Proposition 1.1.]{kim86-1}}
	\label{thm_affinest}
	 Let $G$ be a simply connected Lie group with its Lie algebra $\ggo$, then the following statements are equivalent:
	\begin{enumerate}
		\item the Lie algebra $\ggo$ admits a complete left-symmetric structure;
		\item the Lie group $G$ acts simply transitively on $\RR^n$ by affine transformations.
	\end{enumerate}
\end{theorem}

\begin{remark}\label{flat torsion free}
	Although this is not important for the remainder of this paper, the definition of left-symmetric structures is closely related to complete affine manifolds, which have a rich history on their own. Note that a left-symmetric structure on $\ggo$ defines a left-invariant connection $\nabla$ on the corresponding $1$-connected Lie group $G$ by $\nabla_xy=L_x(y)$ for all $x, y \in \ggo$. The first condition of a left-symmetric structure is equivalent to $\nabla$ being torsion-free, while the second condition is equivalent to the curvature being zero. Moreover, this connection is geodesically complete, that is, every geodesic can defined over the real line, if and only if the left-symmetric structure is complete. 
\end{remark}

In \cite{kim86-1}, Kim used Theorem \ref{thm_affinest} to determine the existence of simply transitive actions on $\RR^n$ via affine transformations for a given Lie group. The author focused on nilpotent Lie algebras admitting a left-symmetric structure and studied their properties, mostly in relation to extensions and the existence of a center. Finally, a classification of 4-dimensional LS whose left multiplication $L_x$ is nilpotent is given up to affine equivalence.

\subsection{NIL-affine actions with $G$ nilpotent}

The case of NIL-affine actions $\rho: G \to \Aff(H)$ when both $G$ and $H$ are nilpotent was studied in \cite{bdd09-1}. They translated the problem to the Lie algebra level by considering the induced map $\varphi=d\rho$ given by
\begin{align}\label{t D parts}
	\varphi : \ggo &\to \aff(\hgo) =  \hgo \rtimes \Der(\hgo)  \\
	x &\mapsto \left(t(x),D(x) \right) \nonumber
\end{align}
with $t: \ggo \to \hgo$ a linear map and $D: \ggo \to \Der(\hgo)$ a morphism of Lie algebras.  Recall that $\aff(\hgo) = \hgo \rtimes \Der(\hgo)$ is equal to the Lie algebra of $\Aff(H)$, with its Lie bracket given by 
\begin{equation}\label{eq: lie bracket aff}
	[(x,D), (x',D')]_{\aff(\hgo)} = ([x, x']_\hgo + D(x') - D'(x), D \circ D^\prime - D^\prime \circ D).
\end{equation}

The main result is stated as follows.
\begin{theorem}{\cite[Theorem 3.1.]{bdd09-1}}
	\label{thm:kd}
	Let $G$ and $H$ be $1$-connected nilpotent Lie groups and $\rho: G \to \Aff(H)$ a representation with corresponding maps $t: \ggo \to \hgo$ and $D: \ggo \to \Der(\hgo)$. The group $G$ acts simply transitive on $H$ if and only if the map $t: \ggo \to \hgo$ is a bijection and $D(x)$ is nilpotent for every $x \in \ggo$. 
\end{theorem}

The authors also linked the existence of such actions to the existence of PLAS, which generalize the notion of left-symmetric structure (also called pre-Lie algebras structures). Before stating this link result, we recall de definition of PLAS.  
\begin{definition}
	\label{def_PLAS}
	Let $\ggo=(V,[\cdot,\cdot]_\ggo)$ and $\hgo=(V, [\cdot,\cdot]_\hgo)$  be two Lie algebras on the same underlying vector space $V$. A \emph{post-Lie algebra structure} (abbreviated as PLAS) on the pair $(\ggo,\hgo)$ is a bilinear product $$\cdot: V \times V \to V$$ satisfying the following identities:  
	\begin{enumerate}
		\item $x \cdot y - y \cdot x= [x, y]_\ggo - [x, y]_\hgo$
		\item $[x, y]_\ggo \cdot z =  x \cdot (y \cdot z) - y \cdot (x \cdot z)$,
		\item $x \cdot [y, z]_\hgo = [x \cdot y, z]_\hgo + [y, x \cdot z]_\hgo$
	\end{enumerate}
	for all $x, y, z \in V$.
\end{definition} \noindent A different way of looking at PLAS as subalgebras of $\aff(\hgo)$ will be presented below. Although PLAS can be defined over any field $k$, we focus on the reals $\RR$ in this paper because of the link with NIL-affine actions. In the special case that $[ x,y ]_\hgo$ is zero for all $x,y\in V$, a post-Lie algebra is nothing else than a left-symmetric algebra on the Lie algebra $\ggo$.

There also exists the notion of a post-Lie algebra, which is a Lie algebra $\hgo$ equipped with a product $\cdot$ satisfying certain conditions, such that condition (1) defines a Lie bracket, see e.g.~\cite{vall7}. To avoid confusion, we do not pursue this point of view, although it is equivalent to working with PLAS as we will do below.

Usually we denote by $L_x: V \to V$ and $R_y: V \to V$ with $x, y \in V$ the left and right multiplication operators of the algebra $(V,\cdot)$, respectively. It is easy to see that condition (3) corresponds to the fact that $L_x$ is a derivation of the Lie algebra $\hgo$ for every $x \in V$, and that condition (2) means that the map $L: V \to \Der(\hgo)$ is a Lie algebra morphism. 

\begin{remark}
	\label{rmk_mult} A PLAS is fully characterized by either the left or the right multiplication map, so often we will just speak about the PLAS given by $L: V \to \Der(\hgo)$. In fact, once we know the Lie algebra $\hgo$ and the left multiplication, we can compute the second Lie bracket via condition (1), a point of view we will exploit in the examples, see Section \ref{sec_ex}. 
\end{remark}

A lot is known about the existence of PLAS on an arbitrary pair of Lie algebras, we refer to \cite{bd16-1} for an overview of these results. In this paper we will focus on the case when $\ggo$ is solvable and $\hgo$ is nilpotent, due to the motivation from NIL-affine actions.

Theorem \ref{thm:kd} was rephrased in therms of PLAS in \cite[Theorem 2.15.]{bdv12-1}.

\begin{theorem}
	Let $G$ and $H$ be $1$-connected nilpotent Lie groups with associated Lie
	algebras $\ggo$ and $\hgo$. There exists a simply transitive NIL-affine action of $G$ on $H$ if and only if there is a Lie algebra $\tilde{\ggo}$, isomorphic to $\ggo$, with the same underlying vector space as $\hgo$ such that the pair of Lie algebras $(\tilde{\ggo},\hgo)$ admits a complete post-Lie algebra structure.
\end{theorem}

Again, the crucial part in this theorem is the correct notion of completeness, which in this case involves the left multiplication, see \cite{bdv12-1}.
\begin{definition}\label{complete PLAS nilpotent case}
A post-Lie algebra structure on a pair $(\ggo,\hgo)$ with $\ggo$ and $\hgo$ both nilpotent is called {\it complete} if all left multiplications are nilpotent. 
\end{definition}

Note that when $\hgo$ is abelian completeness is equivalent to completeness of the LS structure in the sense of Definition \ref{complete_LS}, according to Theorem \ref{thm_completenil}. The question how completeness should be defined in the case of $\ggo$ solvable and $\hgo$ nilpotent is one of the main motivations of this paper. 
%Note that all left multiplications are nilpotent if and only if all right multiplications are nilpotent in this case. Therefore, completeness here implies that the map $\id + R_y$ is bijective.

In \cite{bdd09-1} this theorem is used to show that for any $n \leq 5$ and any solvable and nilpotent $1$-connected Lie groups $G$ and $H$ of dimension $n$, there exists a simply transitive NIL-affine action $\rho: G \to \Aff(H)$. Moreover, the authors initiate the study of the special case when $G$ is abelian, leading to so-called LR-structures.

\subsection{General NIL-affine case}

Recently, the general case when $G$ is solvable and $H$ nilpotent was treated in \cite{DerOri21}. The main result shows that $G$ admits a simply transitive NIL-affine action on $H$ if and only if the semi-simple splitting $\ggo^\prime$ embeds in $\aff(\hgo)$. 
%Recall that the semi-simple splitting is an algebraic construction associated to any solvable Lie algebra, in same sense it is the smallest split type Lie algebra $\ggo'$ such that its nilradical has the same dimension as $\ggo$. 
We do not recall the definition of the semi-simple splitting associated to a solvable Lie algebra, but it is a split solvable Lie algebra $\ggo^\prime = \ngo \rtimes \sg$ containing $\ggo$, where $\ngo$ is the nilradical of $\ggo^\prime$, $\sg$ acts by semi-simple derivations on $\ngo$ and satisfying some additional minimality conditions. We call the nilradical $\ngo$ the nilshadow of the solvable Lie algebra $\ggo$. 

\begin{theorem}\cite[Theorem 3.6]{DerOri21}
	\label{our2}
	A $1$-connected solvable Lie group $G$ acts simply transitive on a nilpotent Lie group $H$ via affine transformations if and only if there exists an injective morphism $\varphi=(t,D) : \ggo' \to \aff(\hgo)$ of the semi-simple splitting $\ggo^\prime = \ngo \rtimes \sg$ such that $\varphi(\ngo)$ consists of nilpotent elements, $\varphi(\sg)$ of semi-simple elements and the map $t|_\ngo: \ngo \to \hgo$ is a bijection.
\end{theorem}
\noindent Here, nilpotent and semi-simple is defined in terms of the linear algebraic group $\aff(\hgo)$, as we will recall in the next section. Although in principle Theorem \ref{our2} is enough to determine all possibilities for simply transitive actions, it remains open how to link this result to other concepts in the literature such as PLAS. 

\begin{remark}
In \cite{DerOri21} the semi-simple splitting was computed for all solvable Lie algebras up to dimension $4$, such that an embedding $\varphi:\ggo^\prime = \ngo \rtimes \sg\to\aff(\hgo)$ is exhibited for any pair $(\ggo,\hgo)$ where $G$ acts simply transitive by NIL-affine action on $H$. However, the corresponding map $\ggo \to \aff(\hgo)$ is not explicitly written in this paper. Given one of these embeddings $\varphi:\ggo^\prime = \ngo \rtimes \sg\to\aff(\hgo)$, the simply transitive NIL-affine action from $G$ on $H$ can be found by composing $\varphi$ with the natural embedding $\it i:\ggo\to\ggo^\prime=\ggo\rtimes\sg$. We illustrate this in the next example.
\end{remark}

\begin{example}\label{r'_{3,0} Nil-affine action}
	
	Consider the Lie algebra $\ggo=\mathfrak r'_{3,0}$ as in \cite[Example 5.1.]{DerOri21}, which is isomorphic to the Lie algebra $\mathfrak e(2)$ of the isometry group of Euclidean $2$-space. The Lie brakets of $\ggo$ are given by 
	$$[e_1,e_2]=e_3, \quad [e_1,e_3]=-e_2.$$ 
	The semi-simple splitting of $\ggo$ is isomorphic to $\ggo'= \RR^3\rtimes \RR  = \langle e_2^\prime,e_3^\prime,e_4^\prime \rangle \rtimes \langle e_1^\prime \rangle$ with Lie brackets $$[e_1^\prime,e_2^\prime]=e_3^\prime, \quad [e_1^\prime,e_3^\prime]=-e_2^\prime \quad \text{and} \quad e_4^\prime\in\zg(\ggo').$$ 

	The following map is an embedding from $\ggo'$ to $\affg(\hg_3)=\hg_3\rtimes\Der(\hg_3)$, 
	$$\varphi(x_1,x_2,x_3,x_4)=\left( (x_2,x_3,x_4),\begin{pmatrix}
		0&-x_1&0\\
		x_1&0&0\\
		\frac{x_3}{2}&-\frac{x_2}{2}&0
	\end{pmatrix}\right),$$
	where we are considering the canonical basis $\{f_1,f_2,f_3\}$ of $\hgo_3$ such that $[f_1,f_2]=f_3$.

	Now, $\ggo$ can be seen as an ideal of $\ggo'$ by identifying $e_1^\prime$ and $e_4^\prime$, that is $\it i:\ggo\to\ggo^\prime$ is given by $\it i (x_1,x_2,x_3)=(x_1,x_2,x_3,x_1)$, and therefore
	$\it \varphi \circ i:\ggo\to \affg(\hg_3)$, given by
	$$\it \varphi \circ i(x_1,x_2,x_3)=\left( (x_2,x_3,x_1),\begin{pmatrix}
		0&-x_1&0\\
		x_1&0&0\\
		\frac{x_3}{2}&-\frac{x_2}{2}&0
	\end{pmatrix}\right)$$
	induce a simply transitive NIL-affine action from $G$ on $H_3$ where $H_3$ denotes the $3$-dimensional Heisenberg Lie group.
\end{example}

\section{From NIL-affine actions to PLAS}
\label{sec:toplas}

\subsection{NIL-affine action: algebraic point of view}

We start by recalling the necessary background about linear algebraic groups, where more details can be found in e.g.~\cite{bore69-1}. A \textit{real linear algebraic group} $G \subset \Gl(n,\RR)$ is a subgroup of $\Gl(n,\RR)$ given as the zero set of a (finite) number of polynomials. A group morphism $\Phi: G \to G^\prime$ between two real linear algebraic groups $G \subset \Gl(n,\RR)$ and $G^\prime \subset \Gl(n^\prime,\RR)$ is called an \textit{algebraic morphism} if the coordinate functions are given by polynomials. An \emph{algebraic isomorphism} is a bijective algebraic morphism such that the inverse is again an algebraic morphism of real linear algebraic groups. 

We call a linear algebraic group $G$ \emph{unipotent} if all its elements are unipotent, meaning that every $X \in G \subset \Gl(n,\RR)$ only has $1$ as complex eigenvalue. Similarly, we call $G$ \emph{semi-simple} if every element is semi-simple, meaning that it is diagonalizable over $\CC$. Every element $X \in \Gl(n,\RR)$ can be decomposed uniquely as $X = X_u X_s$ where $X_u$ is unipotent, $X_s$ is semi-simple and both elements commute. Moreover, if $X \in G$ is an element in a real linear algebraic group, then both elements $X_u, X_s \in G$. If $\Phi: G \to G^\prime$ is an algebraic morphism of real linear algebraic groups, then for every $X \in G$, we have that $\Phi(X_u) = \Phi(X)_u$ and $\Phi(X_s) = \Phi(X)_s$. In particular, the image of a semi-simple or unipotent algebraic group is again semi-simple or unipotent. 

If $G$ is solvable and connected, then the subset of all unipotent elements $U(G)$ forms a unipotent normal subgroup, which is called the \emph{unipotent radical}. Moreover, every such $G$ is isomorphic to $U(G) \rtimes T$ where $T$ is a maximal torus of $G$, i.e.~a maximal subgroup which is diagonalizable over the complex numbers. 

\begin{example}
	Let $H$ be a simply connected and connected nilpotent Lie group, then $H$ has a unique (i.e.~up to algebraic isomorphism) structure as a unipotent real linear algebraic group. Moreover, the automorphism group $\Aut(H)$ is identified with the automorphisms $\Aut(\hgo)$ of the corresponding Lie algebra $\hgo$, which naturally forms a linear algebraic group. In particular, the semi-direct product $\Aff(H) = H \rtimes \Aut(H)$ also carries the structure of a linear algebraic group.
\end{example}

The Lie algebra $\ggo \subset \glg(n,\RR)$ of a real linear algebraic group $G$ is called an \emph{algebraic Lie algebra}. Similarly, we call a morphism $\varphi: \ggo \to \ggo^\prime$ between algebraic Lie algebras $\ggo \subset \glg(n,\RR)$ and $\ggo^\prime \subset \glg(n^\prime,\RR)$ algebraic if it is induced by an algebraic morphism of linear algebraic groups. Given a subalgebra $\ggo \subset \glg(n,\RR)$, there always exists a smallest algebraic Lie algebra $\overline{\ggo}$ containing $\ggo$, called the algebraic closure of $\ggo$. If $\ggo$ is solvable, then we write $\mathfrak{u}_\ggo$ for the Lie algebra corresponding to the unipotent radical of the algebraic closure $\overline{\ggo}$, which consists exactly of the nilpotent elements in $\overline{\ggo}$. Moreover, we can write $\overline{\ggo} = \mathfrak{u}_\ggo + \mathfrak{s}$ where $\mathfrak{s}$ is the Lie algebra corresponding to a maximal torus, although the subalgebra $\mathfrak{s}$ in general is not unique.

In this short example, we illustrate how the properties of a algebraic Lie algebra $\ggo$ depend on the embedding into $\glg(n,\RR)$.

\begin{example}
	\label{ex_2structure}
	
Consider the Lie algebras
\begin{align*}
		\ggo_1 = \left\{ \begin{pmatrix}x & 0 \\ 0 & 0  \end{pmatrix} \mid x \in \RR \right\}, \, \ggo_2 = \left\{ \begin{pmatrix}0 & x \\ 0 & 0  \end{pmatrix} \mid x \in \RR \right\} \subset \glg(2,\RR)
\end{align*} 
given as subalgebras of $\glg(2,\RR)$. Of course $\ggo_1$ and $\ggo_2$ are isomorphic as Lie algebras, as both are $1$-dimensional. It is easy to check that both are algebraic as well. However, they are not isomorphic as algebraic Lie algebras, as the first corresponds to a semi-simple linear algebraic group and the second to a unipotent. It is easy to see that not every linear map on $\ggo_1$ is algebraic, by considering for example the map $D(x) = \frac{1}{2} x$ for $x \in \RR$.
\end{example}

As $\Aff(H)$ is a real linear algebraic group, we can also take the algebraic closure of subalgebras in the algebraic Lie algebra $\aff(\hgo)$. Using this construction, the paper \cite{DerOri21} also contains a method to check whether a given action is simply transitive, by using the Jordan decomposition of elements in $\aff(\hgo)$, based on the following result.

\begin{theorem}\cite[Theorem 3.4]{DerOri21}
	\label{thm_closure}
Let $\hgo$ be a nilpotent Lie algebra and $\varphi: \ggo \to \aff(\hgo)$ a morphism of a solvable Lie algebra $\ggo$. Consider the ideal of nilpotent elements $\mathfrak{u}_{\varphi(\ggo)}$ of the closure of the image $\varphi(\ggo)$ as described above. The corresponding connected solvable Lie group acts simply transitively on $H$ if and only if $\dim(\ggo) = \dim(\hgo)$ and the restriction $\mathfrak{u}_{\varphi(\ggo)} \to \hgo$ of the natural projection on the first component is a bijection. 
\end{theorem}

Since we know that the dimensions $\dim(\hgo) = \dim(\ggo)$ are equal, the latter condition is equivalent to $p(\mathfrak{u}_{\varphi(\ggo)} ) = \hgo$ with $p: \aff(\hgo) \to \hgo$ the projection on the first component.

\begin{remark}
	\label{remark_ss}
Although not explicitly written in this theorem, we can actually say more about the Lie algebra $\overline{\varphi(\ggo)}$ corresponding to a simply transitive action $\rho: G \to \Aff(H)$. Indeed, in the proof of \cite[Theorem 3.5.]{DerOri21} it is shown that the stabilizer of $e \in H$ forms a maximal torus of $\overline{\rho(G)}$, so with corresponding algebraic Lie algebra $\mathfrak{s} \subset \Der(\hgo)$. In particular, we can write $\mathfrak{u}_{\varphi(\ggo)}+\mathfrak{s} = \overline{\varphi(\ggo)}$ for a subalgebra $\mathfrak{s} \subset \Der(\hgo)$. 
\end{remark}

Although Theorem \ref{thm_closure} gives a method to check whether a general NIL-affine action is simply transitive or not, there is not yet a relation with PLAS as mentioned in the introduction for certain special cases. In the following section we show that every simply transitive NIL-affine action indeed induces as PLAS.

\subsection{PLAS as $t$-bijective subalgebras}

For our techniques, a different characterization of post-Lie algebra structure will be necessary, similar to the approach in \cite{seg92-1}. We write $p: \aff(\hgo) \to V$ and $q: \aff(\hgo) \to \Der(\hgo)$ for the projections on the first and the second component respectively. Although $q$ is a morphism of Lie algebras, the map $p$ is only a linear map of vector spaces, hence we write the image as $V$.

There is a relation between post-Lie algebra structures on $(\ggo,\hgo)$ and subalgebras $\tilde\ggo$ of $\aff(\hgo)$ for which the projection $p: \aff(\hgo) \to\hgo$ onto the first factor induces a Lie algebra isomorphism from $\tilde\ggo$ to $\ggo$. To describe the relation, first note that given a PLAS with corresponding left multiplication $L$, we have that $$\tilde{\ggo}= \left\{(X, L_x) \right\} \subset \aff(\hgo)$$ is a subalgebra such that 
the projection $p$ induces an isomorphism between $\tilde{\ggo}$ and $\ggo$. On the other hand, let $\tilde{\ggo} \subset \aff(\hgo)$ be a subalgebra with inclusion map $\varphi = (t,D): \tilde{\ggo} \to \aff(\hgo)$ as in \eqref{t D parts}. If the map $t: \tilde{\ggo} \to V$ forms an isomorphism between $\tilde{\ggo}$ and $\ggo$, then the inverse isomorphism $t^{-1}: V\to\tilde\ggo\subset\aff(\hgo)$ induces a PLAS given by 
\begin{align}\label{eq: left multiplication}
\nonumber	L: V &\to \Der(\hgo) \\ x & \mapsto D(t^{-1}(x)).
\end{align}  
%\textcolor{red}{Check this and compare with Thm 2.15 BDD!!}
As proven for example in \cite[Proposition 2.12]{bdv12-1}, this relation is indeed bijective.

\begin{proposition}\label{subalgebras}
	The maps above form a 1-1 correspondence between the post-Lie algebra structures on $(\ggo,\hgo)$ and subalgebras $\tilde\ggo$ of $\aff(\hgo)$ for which the projection $p: \aff(\hgo) \to\hgo$ onto the first factor induces a Lie algebra isomorphism from $\tilde\ggo$ to $\ggo$.
\end{proposition}
\begin{proof}
	A direct check shows that both constructions above are each others inverse.
\end{proof}

This correspondence motivates the following (more general) definition of subalgebras in $\aff(\hgo)$.

\begin{definition}
	Let $\hgo$ be a Lie algebra with underlying vector space $V$. A Lie algebra $\ggo$ together with a map $\varphi=(t,D) : \ggo \to \aff(\hgo)$ as in (\ref{t D parts}) such that $t: \ggo \to \hgo$ is bijective is called {\it $t$-bijective}. In particular, a subalgebra $\tilde{\ggo} \subset \aff(\hgo)$ is called a {\it $t$-bijective subalgebra} if it is $t$-bijective for the inclusion map $i: \tilde{\ggo} \to \aff(\hgo)$. 
\end{definition}

\begin{remark}\label{PLA}
In the previous definition, we only have one Lie bracket on the vector space $V$. However, starting from a $t$-bijective subalgebra $\tilde{\ggo} \subset \aff(\hgo)$, we can construct a second Lie bracket and corresponding PLAS in the following way. Define the Lie bracket $[\cdot,\cdot]_{\ggo}$ on $V$ by $$[x,y]_{\ggo} = t \left([t^{-1}(x),t^{-1}(y)]_{\tilde{\ggo}} \right).$$ This new Lie algebra $\ggo$ is clearly isomorphic to $\tilde\ggo$ via the map $t$ and by Proposition \ref{subalgebras} we find a PLAS on the pair $(\ggo,\hgo)$. This last construction will be of particular importance to us, and is equivalent to Definition \ref{def_PLAS} above.
\end{remark}

We combine Proposition \ref{subalgebras} and Theorem \ref{thm_closure} to induce a PLAS on $(\ggo,\hgo)$ from a simple transitive NIL-affine action from $G$ to $H$. This construction generalizes the case where $(\ggo,\hgo)$ are both nilpotent Lie algebras as described in \cite[Remark 2.13]{bdv12-1}).

\begin{theorem}\label{from ST to PLAS}
Let $\rho: G \to \Aff(H)$ be a simply transitive NIL-affine action of a solvable Lie group $H$ on a nilpotent Lie group $H$ with corresponding morphism $\varphi=(t,D): \ggo \to \aff(\hgo)$ as in \eqref{t D parts}. Then the Lie algebra $\ggo$ with the morphism $\varphi: \ggo \to \aff(\hgo)$ forms a $t$-bijective subalgebra of $\aff(\hgo)$.
\end{theorem}

\begin{proof}
As the action is simple, the morphism $\varphi$ is always injective. By replacing $\ggo$ by $\varphi(\ggo)$, we can assume that $\ggo$ is a subalgebra of $\aff(\hgo)$, still corresponding to a simply transitive action on $H$, namely of the solvable Lie group $\rho(G) \subset \Aff(H)$. 

Take $\overline{\ggo}$ the algebraic closure of $\ggo$ in $\aff(\hgo)$, then by Theorem \ref{thm_closure} and Remark \ref{remark_ss}, we know that $\overline{\ggo} = \mathfrak{u}_\ggo + \mathfrak{s}$, where $\mathfrak{u}_\ggo$ is the subalgebra consisting of nilpotent elements and $\mathfrak{s} \subset \Der(\hgo)$. Now take $p: \aff(\hgo) \to \hgo$ the projection on the first component, and thus $t$ is the restriction of $p$ to $\ggo$. In particular we have that $p(\mathfrak{s}) = 0$. Since $\mathfrak{s}\cap \ggo = 0$ by \cite[Theorem 3.1.]{DerOri21} and $\overline{\ggo} = \ggo + \mathfrak{s}$ holds by comparing the dimensions, we get that $$t(\ggo) = p(\ggo) = p(\ggo + \mathfrak{s}) = p(\overline{\ggo}) = p(\mathfrak{u}_\ggo + \mathfrak{s} )= p(\mathfrak{u}_\ggo) = \hgo,$$ where the last equality holds because of Theorem \ref{thm_closure}. Because $\dim(\ggo) = \dim(\hgo)$ and $t$ is linear, we thus find that $t$ forms a bijection between $\ggo$ and $\hgo$. 
\end{proof}

\begin{corollary}\label{cor:PLAS}
Every simply transitive action $\rho: G \to \Aff(H)$ with corresponding Lie algebras $\tilde{\ggo}$ and $\hgo$ induces a PLAS on a pair $(\ggo,\hgo)$ with $\tilde{\ggo}$ isomorphic to $\ggo$.
\end{corollary}

\begin{proof}
This follows immediately from the previous theorem and Remark \ref{PLA}.
	\end{proof}
We illustrate the proof with the following example from \cite{DerOri21}. From the tables in that paper, it is immediate to construct the PLAS for every simply transitive action. 

\begin{example}\label{r'_{3,0} PLAS}
	
	Let us consider Example \ref{r'_{3,0} Nil-affine action} again and compute the PLAS associated to that simply transitive NIL-affine action. Recall that $\ggo=\mathfrak r'_{3,0}$ and $\hgo=\hgo_3$.
	
	In order to induce a PLAS, we use Remark \ref{PLA} to define a Lie algebra $\tilde{\ggo}$ isomorphic to $\ggo$ but on the same vector space as $\hgo_3$. Indeed, if we take the basis $f_1, f_2, f_3$ for $\hgo_3$, thus with 	$[f_1,f_2]_{\hgo_3}=f_3$, then it holds that $t(e_1) = f_3, t(e_2) = f_1$ and $t(e_3)= f_2$ by Example \ref{r'_{3,0} Nil-affine action}. In particular, the Lie bracket of $\tilde{\ggo}$ is given by:
	\begin{align*}[f_3,f_1]_{\ggo}&=f_2, \\ [f_3,f_2]_{\ggo}&=-f_1
	\end{align*}
	respectively for the basis $\{f_1,f_2,f_3\}$.
	Now, it follows from Proposition \ref{subalgebras} that $(\tilde{\ggo},\hgo_3)$ admits a PLAS with 
	$$L_x = D(t^{-1}x)= x_1 D(e_2) + x_2 D(e_3) + x_3 D(e_1) = \begin{pmatrix}
		0&-x_3&0\\
		x_3&0&0\\
		\frac{x_2}{2}&-\frac{x_1}{2}&0
	\end{pmatrix},$$ where $x =x_1 f_1 + x_2 f_2 + x_3 f_3 \in \hgo_3$.
\end{example}

From now on, we will often describe a PLAS as a Lie algebra $\hgo = (V,[\cdot,\cdot]_\hgo)$ with a left multiplication $L$, without mentioning the induced Lie bracket $[\cdot,\cdot]_\ggo$ on $V$, see Remark \ref{rmk_mult}.

\begin{remark}
	\label{rmk:lefttoright}
	Note that neither the left multiplications nor the right multiplications are nilpotent in Example \ref{r'_{3,0} PLAS}. Indeed, the left multiplications $L_x$ have eigenvalues $\{0, \pm ix_3\}$ by a direct computation.
	
	 To compute the right multiplication, we write
	$L_x(y)=\begin{pmatrix}
		0&-x_3&0\\
		x_3&0&0\\
		\frac{x_2}{2}&-\frac{x_1}{2}&0
	\end{pmatrix}\begin{pmatrix}
		y_1\\
		y_2\\
		y_3
	\end{pmatrix}$, which is equivalent to $\begin{pmatrix}
		0&0&-y_2\\
		0&0&y_1\\
		-\frac{y_2}{2}&\frac{y_1}{2}&0
	\end{pmatrix}\begin{pmatrix}
		x_1\\
		x_2\\
		x_3
	\end{pmatrix}$. Therefore the maps $R_y=\begin{pmatrix}
		0&0&-y_2\\
		0&0&y_1\\
		-\frac{y_2}{2}&\frac{y_1}{2}&0
	\end{pmatrix}$ are the right multiplication, thus with eigenvalues $\left\{0, \pm\sqrt{ \frac{y_1^2+y_2^2}2}\right\}$.
	Moreover, the map $\id + R_y$ is not bijective for all $y\in\ggo$, as it has eigenvalues $\left\{1, 1\pm \sqrt{\frac{y_1^2+y_2^2}2}\right\}$ which might contain $0$. Therefore completeness as classically defined in Definition \ref{complete_LS} or in \ref{complete PLAS nilpotent case} can not be extended to the general solvable case.
	
	The question what is a necessary and sufficient condition on a given PLAS to induce a simply transitive NIL-affine action naturally arises, and such class of PLAS satisfying this condition should be called complete.

\end{remark}

\section{Two-step nilpotent case}

In this section we characterize PLAS that induce a simply transitive action when $G$ is solvable and $H$ is $2$-step nilpotent. The main idea is to apply a relation between PLAS on a pair $(\ggo,\hgo)$ with $\hgo$ a $2$-step nilpotent Lie algebra and left-symmetric structure on $\ggo$, as we explain in the first part.
 
\subsection{Relation PLAS and LS}
\label{subs:plastols}
The following proposition shows that given a PLAS on the pair $(\ggo,\hgo)$ with $\hgo$ $2$-step nilpotent, one finds a left-symmetric structure on $\ggo$. 
\begin{proposition}\label{from PLAS to LS}
	Let $\hgo$ be a $2$-step nilpotent Lie algebra and $\ggo$ a solvable Lie algebra defined on the same underlying vector space $V$ as $\hgo$. Then every PLAS on $(\ggo,\hgo)$ induces a left-symmetric structure on $\ggo$, by changing the left multiplication $L$ to $\tilde{L}$ given by	\begin{equation}\label{left multiplication PLAS-LS}
		\tilde L_x=L_x+\frac12\ad^{\hgo}_x.
	\end{equation}
\end{proposition}

\begin{proof}
	Assume there exists a PLAS on $(\ggo,\hgo)$, determined by the left multiplication $L:\ggo\to\Der(\hgo)$. We define the new left multiplication $\tilde L:\ggo\to \glg(n,\RR)$ as in Equation (\ref{left multiplication PLAS-LS}) and check the condition in Definition \ref{complete_LS}. For condition (1), we get
	\begin{align*}
		\tilde L_x(y)-\tilde L(y)(x) & =(L_x+\frac12\ad^{\hgo}_x)(y)-(L(y)+\frac12\ad^{\hgo}_y)(x)\\
		& =L_x(y)-L(y)(x)+\ad^{\hgo}_x(y)= [x,y]_{\ggo}
	\end{align*}
for all $x,y \in V$, where the last equality follows from condition $(1)$ in Definition \ref{PLA}. Now, condition $(2)$ follows from condition $(2)$ in Definition \ref{PLA} and the fact that $\tilde L([x,y])= L([x,y])$ since $\ad^{\hgo}_{[x,y]}$ vanishes on a $2$-step nilpotent Lie algebra $\hgo$. Therefore it is immediate that $\tilde L$ determines a left-symmetric structure on $\ggo$.
\end{proof}

This relation can also be described on the level of subalgebras of $\aff(\hgo)$. Indeed, there exists an explicit embedding of $\aff(\hgo)$ into $\aff(V)$, where we consider $V$ as the abelian Lie algebra, first described in \cite{bdv12-1}.

\begin{proposition}\label{plas-LS for 2-step}
	For any $2$-step nilpotent Lie algebra $\hgo$ with underlying vector space $V$, we get that the map $\psi: \aff(\hgo) \to \aff(V) = V \rtimes \glg(V)$ defined as 
	\begin{equation}\label{psi map}
		\psi(x,D) = (x, D + \frac{1}{2} \ad_x)
	\end{equation}
	is an injective Lie algebra morphism. 
\end{proposition}

\begin{proof}	 
Let $(x,D), (x',D')\in\aff(\hgo)$, then we have
\begin{align*}
	\psi([(x,D), (x',D')])& = \psi([x, x'] + D(x') - D'(x), [D,D'])\\
	&=([x, x'] + D(x') - D'(x),[D,D']+\frac12\ad_{[x, x'] + D(x') - D'(x)}),\\
	[\psi(x,D), \psi(x',D')]& = [(x, D + \frac{1}{2} \ad_x),(x', D' + \frac{1}{2} \ad_{x'})]\\
	&=((D + \frac{1}{2} \ad_x)(x')-(D' + \frac{1}{2} \ad_{x'})(x),[D + \frac{1}{2} \ad_x,D' + \frac{1}{2} \ad_{x'}])\\
	&=(D(x')+\frac{1}{2}[x,x']-D'(x)-\frac12[x',x],[D,D']+\frac12[D,\ad_{x'}]+\frac12[\ad_x,D']).
\end{align*}
Since $D \in\Der(\hgo)$, it follows that $[D,\ad_{x^\prime}] = \ad_D(x^\prime)$ and similarly for $D^\prime$. Note also that $\ad_{[x, x']}=0$ as $\hgo$ is $2$-step nilpotent. Therefore, $\psi([(x,D), (x',D')])=	[\psi(x,D), \psi(x',D')]$.
\end{proof}

In particular, as $\psi$ is the identity in the first component, it maps $t$-bijective subalgebras $\tilde{g} \subset \aff(\hgo)$ to $t$-bijective subalgebras of $\aff(V)$. Hence, if $\cdot$ is a PLAS on a pair of Lie algebras with $\hgo$ a $2$-step nilpotent Lie algebra, then one has a corresponding subalgebra $\tilde{\ggo} \subset \aff(\hgo)$ by Proposition \ref{subalgebras}. The map $\psi$ then gives a subalgebra $\psi(\tilde{\ggo})$ which is again $t$-bijective, and thus a corresponding left-symmetric structure. It is immediate that the correspondence induced by $\psi$ is exactly the one given in Proposition \ref{from PLAS to LS}.

The main result of this subsection states that the relation between PLAS and LS in Proposition \ref{from PLAS to LS} preserves simply transitive actions. Consider $\Aff(H)$ as a Lie group, and write $\Aff^0(H)$ for the connected component containing the identity, which is equal to $H \rtimes \Aut^0(H)$ with $\Aut(H)$ the connected component of the identity in $\Aut(H)$. Let $\Psi:\Aff^0(H)\to \Aff(V)$ be the morphism such that $d\Psi=\psi$ as in Proposition \ref{plas-LS for 2-step}. Note that since we assume that the Lie group $G$ is connected, every NIL-affine action $\rho: G \to \Aff(H)$ has an image lying in $\Aff^0(H)$, so it makes sense to define $\Psi \circ \rho$ for such actions.

\begin{theorem}\label{thm_ST}
	Let	$\rho: G \to \Aff(H)$ be a representation, then the corresponding NIL-affine action is simply transitive if and only if the affine action corresponding to $\Psi\circ\rho: G \to \Aff(V)$ is simply transitive.
\end{theorem}

In Example \ref{LS no PLAS} we illustrate that not every complete LS can be found using this construction. We give two different proofs of this theorem, one by using the characterization as in Theorem \ref{thm_closure}, and a second by using the exact form of $\Psi$. The following lemma will be useful in the first proof. 

\begin{lemma}
The map $\psi$ is an algebraic morphism of Lie algebras. 
\end{lemma}

\begin{proof}
First note that any linear map on an abelian Lie algebra $V$, given its unique structure as a unipotent algebraic Lie algebra, is algebraic. One can check this for example using the explicit nilpotent embedding of $\RR$ into $\glg(2,\RR)$ as in Example \ref{ex_2structure} above. Now for any Lie algebra $\hgo$, the adjoint map $\ad:\hgo \to \Der(\hgo)$ is algebraic by \cite[Proposition 10.3.]{hump81-1}. The image $\ad(\hgo)$ is an algebraic abelian subalgebra of $\Der(\hgo)$, corresponding to a unipotent algebraic group. In particular the linear map $x \mapsto \frac{1}{2} x$ is algebraic, and thus it is immediate to see that $\psi$ is algebraic as composition of algebraic maps.
\end{proof}

The proof now follows directly from Theorem \ref{thm_closure}.

\begin{proof}[First proof of Theorem \ref{thm_ST}]
	We will prove this result on the level of the Lie algebras, by considering the induced morphism $\varphi: \ggo \to \aff(\hgo)$ and $\psi \circ \varphi: \ggo \to \aff(V)$. Note that there can only be a simply transitive action if the dimension $\dim(\ggo) = \dim(\hgo)$. So from now on we will assume that the dimensions are equal, because otherwise the result is immediate. Similarly as before, we can also replace $\ggo$ by $\varphi(\ggo)$ to assume that $\ggo \subset \aff(\hgo)$.
	
	Write $p: \aff(\hgo) \to V$ and $q: \aff(V) \to V$ for the projections on the first component. We know that the action corresponding to $\ggo$ is simply transitive if and only if $p(\mathfrak{u}_{\ggo}) = \hgo$, and a similar statement for $\psi(\ggo)$ and $q$. However, as $\psi$ is an algebraic morphism, we know that $\psi(\overline{\ggo}) = \overline{\psi(\ggo)}$, and also that $\mathfrak{u}_{\psi(\ggo)} = \psi(\mathfrak{u}_\ggo)$. Now, by definition of $\psi$ we find that $p \circ \psi = q$, and thus that $q(\mathfrak{u}_{\psi(\ggo)} ) = p(\mathfrak{u}_\ggo)$. Hence it immediately follows that $\ggo$ corresponds to a simply transitive action if and only if $\psi(\ggo)$ corresponds to a simply transitive action.
\end{proof}

The second proof is more explicit, by making some computations with $\psi$ and $\Psi$. Note that in $\Aff(H)$ both $H$ and $\Aut(H)$ can be considered as subgroups, namely every $h \in H$ corresponds to $(h,\id_H)$ and $\alpha\in \Aut(H)$ to $(e_H,\alpha)$. With these assumptions, we can write $(h,\alpha)$ as the product $h \alpha$ for $h \in H$ and $\alpha \in \Aut(H)$ using this identification.
\begin{proof}[Second proof of Theorem \ref{thm_ST}] 
	Note that $\psi(\Der(\hgo)) \subset \Der(V)$ and thus that $\Psi(\Aut^0(H)) \subset \Gl(V)$. The exponential map $\exp_\hgo: \hgo \to H$ on $\hgo$ is bijective, with inverse denoted $\log: H \to \hgo$. To describe the exponential map on $V \rtimes \glg(V)$, we note that we can represent the element $(X,D) \in V \rtimes \glg(V)$ as a matrix $$\begin{pmatrix}D & X \\ 0 & 0 \end{pmatrix} \in \glg(n+1,\RR)$$ after choosing a basis for $V$. For every $x \in \hgo$, we have that $\psi(x) = (x,\frac{1}{2} \ad_x)$ and thus the square of the corresponding matrix will be $0$. In particular, $\exp(\psi(x)) = (x,\id_V + \frac{1}{2}\ad_x)$ and thus for every $X \in H$ we have that $$\Psi(X) = \exp(\psi(\log(X))) = (\log(X), \id_V + \frac{1}{2} \ad_{\log(X)}).$$
	
	Now assume that $\rho: G \to \Aff(H)$ is a NIL-affine action, with $\rho(g) = (T(g),L(g))$ for maps $T: G \to H$ and $L: G \to \Aut(H)$. Then $\rho$ is simply transitive if and only if the map $T$ is a bijection. Now $\Psi\circ \rho = (\tilde{T},\tilde{L})$ is of the form $$\Psi(\rho(g)) = \Psi(T(g),L(g)) = \Psi(T(g)) \Psi(L(g)) = (\log(T(g)), \id_V + \frac{1}{2} \ad_{\log(T(g))}) \Psi(L(g)).$$ In particular, $\tilde{T}(g) = \log(T(g))$, and as $\log: H \to \hgo$ is a bijection, we find that $T$ is a bijection if and only if $\tilde{T}$ is a bijection. Equivalently, the theorem holds.
	\end{proof}

%\begin{remark}\Jonas{Should we include it?}
%	One can also define {\it Post Lie algebra} as a triple $(\hgo,[,]_{\hgo},\cdot)$ satisfying 		
%	\begin{enumerate}
%		\item $[x, y]_\hgo \cdot z=  x \cdot (y \cdot z) - y \cdot (x \cdot z)- (x \cdot y)\cdot z + (y \cdot x )\cdot z$
%		\item $x \cdot [y, z]_\hgo = [x \cdot y, z]_\hgo + [y, x \cdot z]_\hgo$
%	\end{enumerate} 
%We can define a new Lie bracket as $[x,y]_\ggo=x\cdot y-y\cdot x +[x, y]_\hgo$. Then $L$ defines a PLAS on  $(\ggo,\hgo)$ if and only if $\tilde L$ defines a LS $\ggo$.
%\end{remark}

%\begin{corollary}
%The injective Lie algebra morphism defined in Proposition \ref{plas-LS for 2-step} is an algebraic morphism. 
%\end{corollary}

%Let $(\ggo,\hgo)$ is a pair of Lie algebras on a vector space $V$, with $\hgo$ $2$-step nilpotent. As mentioned before, we can see the PLAS as a subalgebra $\tilde{\ggo} \subset \aff(\hgo)$ with $t_{\tilde{\ggo}}$ an isomorphism between $\tilde{\ggo}$ and $\ggo$. By taking $\psi(\tilde{\ggo}) \subset \aff(V)$, we find a $t$-bijective subalgebra in $\aff(V)$ and thus a LS. The following result shows that this map preserves PLAS coming from simply transitive actions.

\subsection{Completeness for PLAS}
\label{sec:complete}

Since every simply transitive NIL-affine action on a $2$-step nil induces a simply transitive action, and therefore a complete LS, then we can now define completeness for a PLAS in the $2$-step nilpotent case using this correspondence.

\begin{definition}\label{complete PLAS 2step nil}
 	A post-Lie algebra structure on a pair $(\ggo,\hgo)$ when $\ggo$ is solvable and $\hgo$ is $2$-steps nilpotent is called {\it complete} if  $$R_y -\frac12\ad^\hgo_y$$ is nilpotent for every $y \in V$.
\end{definition}

\begin{example}\label{example: complete plas}
	In Example \ref{r'_{3,0} PLAS} we had left multiplication $L_x = \begin{pmatrix}
		0&-x_3&0\\
		x_3&0&0\\
		\frac{x_2}{2}&-\frac{x_1}{2}&0
	\end{pmatrix},$ and right multiplication given by
$R_y=\begin{pmatrix}
	0&0&-y_2\\
	0&0&y_1\\
	-\frac{y_2}{2}&\frac{y_1}{2}&0
	\end{pmatrix}.$ On the other hand,  $\ad_y=\begin{pmatrix}
	0&0&0\\
	0&0&0\\
	-y_2&y_1&0
\end{pmatrix}$, thus the map $$R_y -\frac12\ad^\hgo_y = \begin{pmatrix} 	0&0&-y_2\\
0&0&y_1\\
0&0&0 \end{pmatrix} $$ is nilpotent for all $y \in V$.
	Therefore, the PLAS from this example is complete. At this point, we want to remark that $c = -\frac{1}{2}$ is the only real number such that $R_y -c\ad^\hgo_y$ is nilpotent.
\end{example}

\begin{proposition}\label{complete PLAS - complete LS}
Let $\cdot$ be a PLAS on a pair of Lie algebras $(\ggo,\hgo)$ where $\hgo$ is $2$-step nilpotent. The PLAS is complete if and only if the LS induced by $\psi$ above is complete.
\end{proposition}

\begin{proof}
	Recall from Equation \eqref{left multiplication PLAS-LS} that $\tilde L_x=L_x+\frac12\ad^{\hgo}_x$. This immediately implies that $\tilde R_y=R_y-\frac12\ad^{\hgo}_y$, where $\tilde R$ is the right multiplication associated to the left-symmetric structure. Therefore the PLAS $\cdot$ is complete if and only if $\tilde{R}_y$ is nilpotent or thus if and only if the LS is complete.
\end{proof}

In the following example we show that not every complete left-symmetric structure can be found using the construction of Section \ref{subs:plastols}, or thus not every $t$-bijective subalgebra of $\aff(V)$ lies in the image of $\psi$ for some $2$-step nilpotent Lie algebra $\hgo$.

\begin{example}\label{LS no PLAS}
	Consider the Lie algebra $\ggo_\lambda=\mathfrak r_{3,\lambda}$, with Lie brackets  given by 
	$$[e_1,e_2]=e_2, \quad [e_1,e_3]=\lambda e_3.$$ 
	In \cite{DerOri21} it is shown that there exists a simply transitive NIL-affine action of the corresponding Lie group on the Heisenberg group $H_3$ if and only if $\lambda \in \{\pm1\}$.
	
	On the other hand, it is easy to check that the map
	$\varphi:\ggo_\lambda\to \aff(V)$, where 
	$$\varphi(x_1,x_2,x_3)=\left( (x_1,x_2,x_3),\begin{pmatrix}
		0&0&0\\
		0&x_1&0\\
		0&0&\lambda x_1
	\end{pmatrix}\right)$$
	induce a complete left-symmetric structure on $\ggo_\lambda$ where the left multiplication is given by 
	$L_x=\begin{pmatrix}
		0&0&0\\
		0&x_1&0\\
		0&0&\lambda x_1
	\end{pmatrix}.$
	In particular, for any $\lambda \notin \{\pm1\}$, $\ggo_\lambda$ admits a left-symmetric structure, but $(\ggo_\lambda,\hgo_3)$ does not admit any PLAS.
\end{example}

Using Theorem \ref{thm_ST}, we now see that this notion of completeness really characterizes the simply transitive actions.  

\begin{theorem}
Let $\cdot$ be a PLAS on the pair $(\ggo,\hgo)$ with $\hgo$ $2$-step nilpotent. Consider the corresponding subalgebra $\tilde{\ggo} \subset \aff(\hgo)$ with a representation $\rho: G \to \Aff(H)$. The NIL-affine action of $G$ on $H$ is simply transitive if and only if $\cdot$ is complete.
\end{theorem}
 
 \begin{proof}
This follows immediately from Theorem \ref{thm_ST} and Proposition \ref{complete PLAS - complete LS} above.
 \end{proof}
 
%We know from Theorem \ref{from PLAS to LS} that every PLAS $(\ggo,\hgo)$ induces a left-symmetric algebra on $\ggo$, we prove now that a complete PLAS correspond to complete LS. 

%
%Now we combine Proposition \ref{complete PLAS - complete LS} and Theorem \ref{thm_ST} to get the converse of Theorem \ref{from ST to PLAS} for the $2$-step case.

We can summarize the correspondence between simply transitive NIL-affine action and complete PLAS as follows.

\begin{theorem}\label{thm main B}
	Let	$G, H$ be $1$-connected Lie groups, with $G$ solvable and $H$ $2$-step nilpotent.  There exists a simply transitive NIL-affine action $\rho: G \to \Aff(H)$ if and only if there exists a Lie algebra $\tilde{\ggo}$, isomorphic to $\ggo$, defined under the same vector space as $\hgo$ and the pair $ (\tilde{\ggo},\hgo)$ admits a complete PLAS. 
\end{theorem}

%\begin{proposition}
%	A PLAS on $(\ggo,\hgo)$ is complete if and only if $R_y -\frac12\ad^\hgo_y$ is nilpotent.
%\end{proposition}
%\begin{proof}
%	A PLAS on $(\ggo,\hgo)$ is complete if $\id+R_y -\frac12\ad^\hgo_y$ is bijective. This is equivalent to $\id+\tilde R_y$ being bijective. Then \cite[Theorem 1]{seg92-1} shows that this is equivalent to $\tilde R_y$ being nilpotent, that is, $R_y -\frac12\ad^\hgo_y$ is nilpotent.
%\end{proof}

Finally, we check that our definition corresponds with the known definitions of completeness. If $\hgo$ is abelian, then this is clear, since then $\ad_y^\hgo = 0$ for all $y \in V$, and thus a PLAS is complete if and only if the right multiplication is nilpotent. The statement then follows from Theorem \ref{thm_affinest}. For the case that both $\ggo$ and $\hgo$ are nilpotent Lie algebras, we first need the following easy lemma.

\begin{lemma}
	Let $\hgo$ be a nilpotent Lie algebra. If $D:\hgo\to\hgo$ is a nilpotent derivation, then $D+\ad^\hgo_x$ is also nilpotent for any $x\in V$.
\end{lemma}
\begin{proof}
We prove this via induction on the nilpotency class of $\hgo$. For abelian Lie algebras, the statement is trivial as $\ad^\hgo_x = 0$ in that case. Now assume that statement holds for nilpotency class $c$, and that $\hgo$ is a Lie algebra of nilpotency class $c+1$. 

If $\gamma_c(\hgo)$ is step $c$ of the lower central series, then we consider the quotient algebra $\overline{\hgo} = \faktor{\hgo}{\gamma_c(\hgo)}$, which is a Lie algebra of nilpotency class $c$. The derivation $D$ induces a derivation of the quotient $\overline{\hgo}$, written as $\overline{D}$, and cleary $\overline{D}$ is again nilpotent. From the induction hypothesis, we know that the induced map $$\overline{D + \ad^\hgo_x} = \overline{D} + \ad^{\overline{\hgo}}_{\overline{x}}$$ on $\overline{\hgo}$ is nilpotent, where $\overline{x}$ is the projection of the element $x$ in $\overline{\hgo}$. So there exists some $k > 0$ such that $\left(D + \ad^\hgo_x\right)^k(\gamma_c(\hgo)) \subset \gamma_c(\hgo)$. Now, as $\ad^\hgo_x(\gamma_c(\hgo)) = 0$ and the restriction of $D$ to $\gamma_c(\hgo)$ is nilpotent, the result follows.
\end{proof}
\begin{proposition}
	Let $\cdot$ be a PLAS on the pair $(\ggo,\hgo)$ with $\ggo$ nilpotent and $\hgo$ $2$-step nilpotent. The map $R_y - \frac{1}{2}\ad_y^\hgo$ is nilpotent for all $y \in V$ if and only if the map $L_x$ is nilpotent for all $x\in V$.
\end{proposition}
\begin{proof}
The first condition is equivalent to $\cdot$ being complete. From Proposition \ref{complete PLAS - complete LS} we have that this is equivalent to the left-symmetric structure determined by $\tilde L_x=L_x+\frac12\ad^{\hgo}_x$	being complete. Since $\ggo$ is nilpotent then this is equivalent to $\tilde L_x$ being nilpotent, see Theorem \ref{thm_completenil}. Now, the statement follows from the Lemma above.
\end{proof}

For all known examples of complete PLAS on a pair $(\ggo,\hgo)$ with both $\ggo$ and $\hgo$ nilpotent, we see that the right multiplication $R_y$ is also nilpotent. This raises the following question.

\begin{question}
Is a PLAS on a pair of nilpotent Lie algebras $(\ggo, \hgo)$ with underlying vector space $V$ complete if and only if the right multiplications $R_y$ are nilpotent for all $y \in V$?
\end{question}
\noindent This shows that the understanding of complete PLAS, even in the case when $\ggo$ is nilpotent, is not yet completely understood.

\section{Examples}
\label{sec_ex}
%Using Corollary \ref{plas-LS for 2-step} above, every PLAS on $(\ggo,\hgo)$ with $\hgo$ $2$-step nilpotent induces a LS on $\ggo$. Moreover, the PLAS is complete if and only if the LS is complete. In the next example we shows that the converse is not true, that is there are examples of LS that are not induced by a PLAS.

As an illustration of the main results, we exhibit in this section a complete PLAS for all pairs $(\ggo,\hgo)$ of Lie algebras with $\hgo$ $2$-step nilpotent and $\ggo$ solvable and not nilpotent up to dimension $4$. Note that the case with $\ggo$ nilpotent was already treated in \cite{bdd09-1} up to dimension $5$, which is the reason we restrict ourselves to the non-nilpotent examples. Note that in dimension $3$ this implies that $\hgo=\hgo_3$ while in dimension $4$ the only Lie algebra to be considered is $\hgo=\hgo_3\times\RR$. We illustrate that completeness is not well understood in higher nilpotency class via Example \ref{dim4_not enough}, showing that we indeed have to restrict ourselves to $\hgo$ being $2$-step nilpotent.

In Tables \ref{table: dim3} and \ref{table: dim4} we give for every $\ggo$ under consideration an embedding from $\ggo$ to $\aff(\hgo)$ in the second column, as explained in Example \ref{r'_{3,0} Nil-affine action}. The third and fourth columns show the left and right multiplication maps respectively. Recall that the left multiplication is computed using \eqref{eq: left multiplication}, and then the right multiplication can be obtained from the left one as in Remark \ref{rmk:lefttoright}. The final column then shows for which examples the new definition for completeness is needed.

\subsection{Dimension three}

An element $x=(x_1,x_2,x_3)^T \in V$ is given in the canonical basis of $\hgo_3$, with $[f_1,f_2]_{\hgo_3} = f_3$. In this basis we have that, $\frac12\ad_x=\begin{pmatrix}
	0&0&0\\
	0&0&0\\
	-\frac{x_2}{2}&\frac{x_1}{2}&0
\end{pmatrix}$. Recall that the Lie bracket in $\aff(\hgo_3)$ is given by \eqref{eq: lie bracket aff}, and the embedding of $\ggo$ into $\aff(\hgo_3)$ is described via the two maps $D: \ggo \to \Der(\hgo)$ and $t: \ggo \to \hgo$. 

We recall the Lie bracket of $3$-dimensional Lie algebras we will need in table \ref{table: dim3} in a basis $e_1, e_2, e_3$.
\begin{align*}
	\mathfrak h_3 &: [e_1, e_2] = e_3 \\ 
	\mathfrak r_3 &: [e_1, e_2] = e_2, [e_1, e_3] = e_2 +e_3 \\  
	\mathfrak r_{3,\lambda} &: [e_1, e_2] = e_2, [e_1, e_3] = \lambda e_3 \\
	\mathfrak r'_{3,\gamma} &: [e_1, e_2] = \lambda e_2 - e_3, [e_1, e_3] = e_2 + \lambda e_3 
\end{align*}

Finally, we exhibit a PLAS for all pairs $(\ggo,\hgo_3)$, coming from simply transitive NIL-affine actions in dimension $3$ in Table \ref{table: dim3}.
\begin{table}[h]
	\begin{tabular}{|c|c|c|c|c|}
		\hline
		$\ggo$ & embedding in $\aff(\hgo_3)$& $L_y = D(t^{-1}(y))$ & $R_z$ & $R_z$ nilpotent?\\
		\hline
		
		$\mathfrak r'_{3,0}$ & $(x_2,x_3,x_1),\begin{pmatrix}
			0&-x_1&0\\
			x_1&0&0\\
			\frac{x_3}{2}&-\frac{x_2}{2}&0
		\end{pmatrix}$ &  
		$\begin{pmatrix}
			0&-y_3&0\\
			y_3&0&0\\
			\frac{y_2}{2}&-\frac{y_1}{2}&0
		\end{pmatrix}$ &
		$\begin{pmatrix}
			0&0&-z_2\\
			0&0&z_1\\
			-\frac{z_2}{2}&\frac{z_1}{2}&0
		\end{pmatrix}$ & \texttimes\\
		\hline
		
		$ \mathfrak r_{3,-1}$&$(x_2,x_3,x_1),\begin{pmatrix}
			x_1&0&0\\
			0&x_1&0\\
			\frac{x_3}{2}&-\frac{x_2}{2}&0
		\end{pmatrix}$ &
		$\begin{pmatrix}
			y_3&&0\\
			0&-y_3&0\\
			\frac{y_2}{2}&-\frac{y_1}{2}&0
		\end{pmatrix}$ &
		$\begin{pmatrix}
			0&0&z_1\\
			0&0&-z_2\\
			-\frac{z_2}{2}&\frac{z_1}{2}&0
		\end{pmatrix}$&\texttimes\\
		\hline
		
		$ \mathfrak r_{3,1}$&
		$(x_1,x_2,x_3),\begin{pmatrix}
			0&0&0\\
			0&x_1&0\\
			\frac{x_3}{2}&-\frac{x_2}{2}&x_1
		\end{pmatrix}$ &
		$\begin{pmatrix}
			0&&0\\
			0&y_1&0\\
			\frac{y_2}{2}&-\frac{y_1}{2}&y_1
		\end{pmatrix}$ &
		$\begin{pmatrix}
			0&0&0\\
			z_2&0&0\\
			-\frac{z_2}{2}+z_3&\frac{z_1}{2}&0
		\end{pmatrix}$&\checkmark\\
		\hline
		
		$ \mathfrak r_3$&$(x_1,x_3,x_2),\begin{pmatrix}
			0&0&0\\
			0&x_1&0\\
			0&0&x_1
		\end{pmatrix}$ &
		$\begin{pmatrix}
			0&0&0\\
			0&y_1&0\\
			0&0&y_1
		\end{pmatrix}$ &
		$\begin{pmatrix}
			0&0&0\\
			z_2&0&0\\
			z_3&0&0
		\end{pmatrix}$&\checkmark\\
		\hline
	\end{tabular}
	\caption{Complete PLAS in dimension $3$}
\end{table}\label{table: dim3}

Note that in Table \ref{table: dim3} the right multiplications for the pairs $(\mathfrak r'_{3,0},\mathfrak h_3)$ and $(\mathfrak r_{3,-1},\mathfrak h_3)$ are not nilpotent showing that the extra term in the definition of completeness is necessary.

\subsection{Dimension four}
Similarly, we recall the Lie bracket of $4$-dimensional solvable Lie algebras and then we exhibit in Table \ref{table: dim4} a complete PLAS for all pairs $(\ggo,\hgo_3\times\RR)$.
% Finally, we illustrate with an example that in the $3$-step nilpotent case the questions how to define completeness remains open.
\begin{align*}
%	\mathfrak h_3 &: [e_1, e_2] = e_3 \\ 
	\mathfrak n_4 &: [e_1, e_2] = e_3, [e_1, e_3] = e_4\\
%	\mathfrak r_3 &: [e_1, e_2] = e_2, [e_1, e_3] = e_2 +e_3 \\  
%	\mathfrak r_{3,\lambda} &: [e_1, e_2] = e_2, [e_1, e_3] = \lambda e_3 \\
%	\mathfrak r'_{3,\gamma} &: [e_1, e_2] = \lambda e_2 - e_3, [e_1, e_3] = e_2 + \lambda e_3 \\
%	\mathfrak r_2\mathfrak r_2  &: [e_1, e_2] = e_2, [e_3, e_4] = e_4  \\
%	\mathfrak r'_2  &: [e_1, e_3] = e_3, [e_1, e_4] = e_4 ,[e_1, e_3] = e_4, [e_2, e_4] = -e_3   \\
	\mathfrak r_4 &: [e_1, e_2] = e_2, [e_1, e_3] = e_2 + e_3, [e_1, e_4] = e_3 + e_4 \\
	\mathfrak r_{4,\mu} &: [e_1, e_2] = e_2, [e_1, e_3] = \mu e_3, [e_1, e_4] = e_3 + \mu e_4 \\
	\mathfrak r_{4,\mu,\lambda} &: [e_1, e_2] = e_2, [e_1, e_3] = \mu e_3, [e_1, e_4] = \lambda e_4\\
	\mathfrak r'_{4,\gamma,\delta} &: [e_1, e_2] = \gamma e_2, [e_1, e_3] = \delta e_3 -e_4, [e_1, e_4] = e_3 +\delta e_4 \\
	\mathfrak d_4 &: [e_1, e_2] = e_2, [e_1, e_3] = e_3, [e_2, e_3] = e_4 \\
	\mathfrak d_{4,\lambda} &: [e_1, e_2] = \lambda e_2, [e_1, e_3] =(1-\lambda) e_3, [e_1, e_4] = e_4, [e_2, e_3] = e_4 \\
	\mathfrak d'_{4,\lambda} &: [e_1, e_2] = \lambda e_2-e_3, [e_1, e_3] = e_2+\lambda e_3, [e_1, e_4] =-2\lambda e_4, [e_2, e_3] = e_4 \\		
	\mathfrak h_{4} &: [e_1, e_2] = e_2, [e_1, e_3] = e_2+e_3, [e_1, e_4] = 2e_4, [e_2, e_3] = e_4\\
\end{align*}
 Now we exhibit in Table \ref{table: dim4} a complete PLAS for all pairs $(\ggo,\hgo_3\times\RR)$, with the same information as above. For this, we again take the standard basis on $\hgo_3 \times \RR$, and the computations are done in a similar fashion.
 
{\tiny
	\begin{table}[h!]
		\begin{tabular}{|c|c|c|c|c|}
			\hline
			$\ggo$ & embedding in $\aff(\hgo_3\times\RR)$& $L_y=D(t^{-1}(y))$ & $R_z$ & $R_z$ nilpotent?\\
			\hline
			
			$ \mathfrak r'_{3,0}\times\RR$& $(x_2,x_3,x_1,x_4),\begin{pmatrix}
				0&-x_1&0&0\\
				x_1&0&0&0\\
				\frac{x_3}{2}&-\frac{x_2}{2}&0&0\\
				0&0&0&0
			\end{pmatrix}$ &  
			$\begin{pmatrix}
				0&-y_3&0&0\\
				y_3&0&0&0\\
				\frac{y_2}{2}&-\frac{y_1}{2}&0&0\\
				0&0&0&0
			\end{pmatrix}$ &
			$\begin{pmatrix}
				0&0&-z_2&0\\
				0&0&z_1&0\\
				-\frac{z_2}{2}&\frac{z_1}{2}&0&0\\
				0&0&0&0
			\end{pmatrix}$ & \texttimes\\
			\hline
			
			$\mathfrak r_{3,-1}\times\RR$ & $(x_2,x_3,x_1,x_4),\begin{pmatrix}
				x_1&0&0&0\\
				0&-x_1&0&0\\
				\frac{x_3}{2}&-\frac{x_2}{2}&0&0\\
				0&0&0&0
			\end{pmatrix}$ &  
			$\begin{pmatrix}
				y_3&0&0&0\\
				0&-y_3&0&0\\
				\frac{y_2}{2}&-\frac{y_1}{2}&0&0\\
				0&0&0&0
			\end{pmatrix}$ &
			$\begin{pmatrix}
				0&0&-z_1&0\\
				0&0&z_2&0\\
				-\frac{z_2}{2}&\frac{z_1}{2}&0&0\\
				0&0&0&0
			\end{pmatrix}$ & \texttimes\\
			\hline

			$ \mathfrak r_{3,1}\times\RR$&
			$(x_1,x_2,x_3,x_4),\begin{pmatrix}
				0&0&0&0\\
				0&x_1&0&0\\
				\frac{x_3}{2}&-\frac{x_2}{2}&x_1&0\\
				0&0&0&0
			\end{pmatrix}$ &
			$\begin{pmatrix}
				0&0&0&0\\
				0&y_1&0&0\\
				\frac{y_2}{2}&-\frac{y_1}{2}&y_1&0\\
				0&0&0&0
			\end{pmatrix}$ &
			$\begin{pmatrix}
				0&0&0&0\\
				z_2&0&0&0\\
				-\frac{z_2}{2}+z_3&\frac{z_1}{2}&0&0\\
				0&0&0&0
			\end{pmatrix}$&\checkmark\\
			\hline
			
			$ \mathfrak r_{3,0}\times\RR$&
			$(x_1,x_3,x_4,x_2),\begin{pmatrix}
				0&0&0&0\\
				0&0&0&0\\
				\frac{x_3}{2}&-\frac{x_1}{2}&0&0\\
				0&0&0&x_1
			\end{pmatrix}$ &
			$\begin{pmatrix}
				0&0&0&0\\
				0&0&0&0\\
				\frac{y_2}{2}&-\frac{y_1}{2}&0&0\\
				0&0&0&y_1
			\end{pmatrix}$ &
			$\begin{pmatrix}
				0&0&0&0\\
				0&0&0&0\\
				-\frac{z_2}{2}&\frac{z_1}{2}&0&0\\
				z_4&0&0&0
			\end{pmatrix}$&\checkmark\\
			\hline
			
					$\mathfrak r'_{4,\gamma,0}$ &  $(x_3,x_4,x_1,x_2),\begin{pmatrix}
				0 & x_1 & 0 & 0 \\
				-x_1 & 0 & 0 & 0 \\
				\frac{x_4}{2} & -\frac{x_3}{2} & 0 & 0 \\
				0 & 0 & 0 & \gamma x_1 \\
			\end{pmatrix}$ &
			$\begin{pmatrix}
				0&y_3&0&0\\
				-y_3&0&0&0\\
				\frac{y_2}{2}&-\frac{y_1}{2}&0&0\\
				0&0&0&\gamma y_3
			\end{pmatrix}$ &
			$\begin{pmatrix}
				0&0&z_2&0\\
				0&0&-z_1&0\\
				-\frac{z_2}{2}&\frac{z_1}{2}&0&0\\
				0&0&\gamma z_4&0
			\end{pmatrix}$ & \texttimes\\
			\hline
			
			$\mathfrak r'_{4,2\delta,\delta}$ &  $(x_3,x_4,x_2,x_1),\begin{pmatrix}
				\delta x_1 & x_1 & 0 & 0 \\
				-x_1 & \delta x_1 & 0 & 0 \\
				\frac{x_4}{2} & -\frac{x_3}{2} & 2\delta x_1 & 0 \\
				0 & 0 & 0 & 0 \\
			\end{pmatrix}$ &
			$\begin{pmatrix}
				\delta y_4 & y_4 & 0 & 0 \\
				-y_4 & \delta y_4 & 0 & 0 \\
				\frac{y_2}{2} & -\frac{y_1}{2} & 2\delta y_4 & 0 \\
				0 & 0 & 0 & 0
			\end{pmatrix}$ &
			$\begin{pmatrix}
				0&0&0&\delta z_1+z_2\\
				0&0&0&-z_1+\delta z_2\\
				-\frac{z_2}{2}&\frac{z_1}{2}&0&2\delta z_3\\
				0&0&0&0
			\end{pmatrix}$ &\checkmark\\
			\hline
			
			$\begin{matrix}\mathfrak r_{4,-1,\lambda}\\ \\ -1\leq\lambda<0
			\end{matrix}$
		%	$\mathfrak r_{4,-1,\lambda}$ $-1\leq\lambda<0$  
		& $(x_2,x_3,x_1,x_4),\begin{pmatrix}
				x_1 & 0 & 0 & 0 \\
				0 & -x_1 & 0 & 0 \\
				\frac{x_3}{2} & -\frac{x_2}{2} & 0 & 0 \\
				0 & 0 & 0 & \lambda x_1 \\
			\end{pmatrix}$ &
			$\begin{pmatrix}
				y_3&0&0&0\\
				0&y_3&0&0\\
				\frac{y_2}{2}&-\frac{y_1}{2}&0&0\\
				0&0&0&\lambda y_3
			\end{pmatrix}$ &
			$\begin{pmatrix}
				0&0&-z_2&0\\
				0&0&z_1&0\\
				-\frac{z_2}{2}&\frac{z_1}{2}&0&0\\
				0&0&\delta z_4&0
			\end{pmatrix}$ & \texttimes\\
			\hline
			
			$\begin{matrix}\mathfrak r_{4,\mu,\mu}\\  -1<\mu\leq1 \\ \mu \neq 0\end{matrix}$ &  $(x_1,x_3,x_4,x_2),\begin{pmatrix}
				0 & 0 & 0 & 0 \\
				0 & \mu x_1 & 0 & 0 \\
				\frac{x_3}{2} & -\frac{x_1}{2} & \mu x_1 & 0 \\
				0 & 0 & 0 & x_1 \\
			\end{pmatrix}$ &
			$\begin{pmatrix}
				0 & 0 & 0 & 0 \\
				0 & \mu y_1 & 0 & 0 \\
				\frac{y_2}{2} & -\frac{y_1}{2} & \mu y_1 & 0 \\
				0 & 0 & 0 & y_1 
			\end{pmatrix}$ &
			$\begin{pmatrix}
				0&0&0&0\\
				\mu z_2&0&0&0\\
				\mu z_3-\frac{z_2}{2}&\frac{z_1}{2}&0&0\\
				z_4&0&0&0
			\end{pmatrix}$ & \checkmark\\
			\hline
			
			$\begin{matrix}\mathfrak r_{4,\mu,-\mu} \\-1 < \mu < 0\end{matrix}$ &  
			$(x_3,x_4,x_1,x_2),\begin{pmatrix}
				\mu x_1 & 0 & 0 & 0 \\
				0 & -\mu x_1 & 0 & 0 \\
				\frac{x_4}{2} & -\frac{x_3}{2} & 0 & 0 \\
				0 & 0 & 0 & x_1 \\
			\end{pmatrix}$ &
			$\begin{pmatrix}
				\mu y_3 & 0 & 0 & 0 \\
				0 & -\mu y_3 & 0 & 0 \\
				\frac{y_2}{2} & -\frac{y_1}{2} & 0 & 0 \\
				0 & 0 & 0 & y_3
			\end{pmatrix}$ &
			$\begin{pmatrix}
				0&0&\mu z_1&0\\
				0&0&-\mu z_2&0\\
				-\frac{z_2}{2}&\frac{z_1}{2}&0&0\\
				0&0&z_4&0
			\end{pmatrix}$ &\texttimes\\
			\hline
			
			$\begin{matrix}\mathfrak r_{4,\mu,1} \\-1 < \mu <1 \\ \mu \neq 0\end{matrix}$ & 
			$(x_2,x_1,x_4,x_3),\begin{pmatrix}
				x_1 & 0 & 0 & 0 \\
				0 & 0 & 0 & 0 \\
				\frac{x_1}{2} & -\frac{x_2}{2} & x_1 & 0 \\
				0 & 0 & 0 & \mu x_1 \\
			\end{pmatrix}$ &
			$\begin{pmatrix}
				y_2& 0 & 0 & 0 \\
				0 & 0 & 0 & 0 \\
				\frac{y_2}{2} & -\frac{y_1}{2} & y_2 & 0 \\
				0 & 0 & 0 & \mu y_2
			\end{pmatrix}$ &
			$\begin{pmatrix}
				0&z_1&0&0\\
				0&0&0&0\\
				-\frac{z_2}{2}&\frac{z_1}{2}+z_3&0&0\\
				0&\mu z_4&0&0
			\end{pmatrix}$ &\checkmark\\
			\hline
			
			$\begin{matrix}\mathfrak r_{4,\mu,1-\mu} \\ \\ 0<\mu<\frac12\end{matrix}$ & 
			$(x_3,x_4,x_2,x_1),\begin{pmatrix}
				\mu x_1 & 0 & 0 & 0 \\
				0 & (1-\mu)x_1 & 0 & 0 \\
				\frac{x_4}{2} & -\frac{x_3}{2} & x_1 & 0 \\
				0 & 0 & 0 & 0 \\
			\end{pmatrix}$ &
			$\begin{pmatrix}
				\mu y_4& 0 & 0 & 0 \\
				0 & (1-\mu)y_4 & 0 & 0 \\
				\frac{y_2}{2} & -\frac{y_1}{2} & y_4 & 0 \\
				0 & 0 & 0 & 0
			\end{pmatrix}$ &
			$\begin{pmatrix}
				0&0&0&\mu z_1\\
				0&0&0&(1-\mu)z_2\\
				-\frac{z_2}{2}&\frac{z_1}{2}&0&z_3\\
				0&0&0&0
			\end{pmatrix}$ &\checkmark\\
			\hline
			
			$\begin{matrix}\mathfrak r_{4,\mu,1+\mu} \\ -1<\mu<0 \\ \mu\neq-\frac12\end{matrix}$ &  
			$(x_2,x_3,x_4,x_1),\begin{pmatrix}
				x_1 & 0 & 0 & 0 \\
				0 & \mu x_1 & 0 & 0 \\
				\frac{x_3}{2} & -\frac{x_2}{2} & (1+\mu)x_1 & 0 \\
				0 & 0 & 0 & 0 \\
			\end{pmatrix}$ &
			$\begin{pmatrix}
				 y_4& 0 & 0 & 0 \\
				0 & \mu y_4 & 0 & 0 \\
				\frac{y_2}{2} & -\frac{y_1}{2} & (1+\mu)y_4 & 0 \\
				0 & 0 & 0 & 0
			\end{pmatrix}$ &
			$\begin{pmatrix}
				0&0&0& z_1\\
				0&0&0&\mu z_2\\
				-\frac{z_2}{2}&\frac{z_1}{2}&0&(1+\mu)z_3\\
				0&0&0&0
			\end{pmatrix}$ &\checkmark\\
			\hline
			
			$\mathfrak r\mathfrak r_3$  &  
			$(x_1,x_3,x_2,x_4),\begin{pmatrix}
				0&0&0&0\\
				0&x_1&0&0\\
				0&0&x_1&0\\
				0&0&0&0
			\end{pmatrix}$ &
			$\begin{pmatrix}
				0&0&0&0\\
				0&y_1&0&0\\
				0&0&y_1&0\\
				0&0&0&0
			\end{pmatrix}$ &
			$\begin{pmatrix}
				0&0&0&0\\
				z_2&0&0&0\\
				z_3&0&0&0\\
				0&0&0&0
			\end{pmatrix}$&\checkmark\\
			\hline
			
			$\mathfrak r_{4,\lambda}$  & 
			$(x_1,x_4,x_3,x_2),\begin{pmatrix}
				0&0&0&0\\
				0& \lambda x_1&0&0\\
				0&0&\lambda x_1&0\\
				0&0&0&x_1
			\end{pmatrix}$ &
			$\begin{pmatrix}
				0&0&0&0\\
				0&\lambda y_1&0&0\\
				0&0&\lambda y_1&0\\
				0&0&0&y_1
			\end{pmatrix}$ &
			$\begin{pmatrix}
				0&0&0&0\\
				\lambda z_2&0&0&0\\
				\lambda z_3&0&0&0\\
				z_4&0&0&0
			\end{pmatrix}$&\checkmark\\
			\hline
			
			$\mathfrak d_4$  & 
			$(x_2,x_3,x_4,x_1),\begin{pmatrix}
				x_1&0&0&0\\
				0&-x_1&0&0\\
				0&0&0&0\\
				0&0&0&0
			\end{pmatrix}$ &
			$\begin{pmatrix}
				y_4&0&0&0\\
				0&-y_4&0&0\\
				0&0&0&0\\
				0&0&0&0
			\end{pmatrix}$ &
			$\begin{pmatrix}
				0&0&0&z_1\\
				0&0&0&-z_2\\
				0&0&0&0\\
				0&0&0&0
			\end{pmatrix}$&\checkmark\\
			\hline
			
			$\mathfrak d_{4,\lambda}$  & 
			$(x_2,x_3,x_4,x_1),\begin{pmatrix}
				\lambda x_1&0&0&0\\
				0&(1-\lambda)x_1&0&0\\
				0&0&x_1&0\\
				0&0&0&0
			\end{pmatrix}$ &
			$\begin{pmatrix}
				\lambda y_4&0&0&0\\
				0&(1-\lambda)y_4&0&0\\
				0&0&y_4&0\\
				0&0&0&0
			\end{pmatrix}$ &
			$\begin{pmatrix}
				0&0&0&\lambda z_1\\
				0&0&0&(1-\lambda)z_2\\
				0&0&0&z_3\\
				0&0&0&0
			\end{pmatrix}$&\checkmark\\
			\hline
			
			$\mathfrak d'_{4,\lambda}$  &  
			$(x_2,x_3,x_4,x_1), \begin{pmatrix}
				\lambda x_1 & x_1 & 0 & 0 \\
				-x_1 & \lambda x_1 & 0 & 0 \\
				0 & 0 & 2\lambda x_1 & 0 \\
				0 & 0 & 0 & 0 \\
			\end{pmatrix}$ &
			$\begin{pmatrix}
				\lambda x_4 & x_4 & 0 & 0 \\
				-x_4 & \lambda x_4 & 0 & 0 \\
				0 & 0 & 2\lambda x_4 & 0 \\
				0 & 0 & 0 & 0 \\
			\end{pmatrix}$ &
			$\begin{pmatrix}
				0&0&0&\lambda z_1+z_2\\
				0&0&0&-z_1+\lambda z_2\\
				0&0&0&2\lambda z_3\\
				0&0&0&0
			\end{pmatrix}$&\checkmark\\
			\hline
			
			$\mathfrak{r}_{4}$  &  
			$(x_1,x_4,x_2,x_3), \begin{pmatrix}
				0 & x_1 & 0 & 0 \\
				0 &  x_1 & 0 & 0 \\
				0 & -x_1 &  x_1 & x_1 \\
				0 & x_1 & 0 & x_1 \\
			\end{pmatrix}$ &
			$\begin{pmatrix}
				0 & y_1 & 0 & 0 \\
				0 &  y_1 & 0 & 0 \\
				0 & -y_1 &  y_1 & y_1 \\
				0 & y_1 & 0 & y_1 \\
			\end{pmatrix}$ &
			$\begin{pmatrix}
				0&0&0&0\\
				z_2&0&0&0\\
				-z_2+z_3+z_4&0&0&0\\
				z_2+z_4&0&0&0
			\end{pmatrix}$&\checkmark\\
			\hline
			
			$\mathfrak{h}_{4}$  & 
			$(x_3,-x_2,x_4,x_1), \begin{pmatrix}
				x_1 & 0 & 0 & 0 \\
				-x_1 &  x_1 & 0 & 0 \\
				0 & 0 &  2x_1 & 0 \\
				0 & 0 & 0 & 0 \\
			\end{pmatrix}$ &
			$\begin{pmatrix}
				y_4 & 0 & 0 & 0 \\
				-y_4 &  y_4 & 0 & 0 \\
				0 & 0 &  2y_4 & 0 \\
				0 & 0 & 0 & 0 \\
			\end{pmatrix}$ &
			$\begin{pmatrix}
				0&0&0&z_1\\
				0&0&0&-z_1+z_2\\
				0&0&0&2z_3\\
				0&0&0&0
			\end{pmatrix}$&\checkmark\\
			\hline
		\end{tabular}
		\caption{Complete PLAS in dimension $4$}
	\end{table}\label{table: dim4}
}

\

Finally, we illustrate now that the notion  completeness as in Definition \ref{complete PLAS 2step nil} is not enough to induce a simply transitive NIL-affine action when $\hgo$ is $3$-step nilpotent (or more).

\begin{example}\label{dim4_not enough}
Let $\ggo=\mathfrak r_{4,-1/2,1/2}$ and $\hgo=\ngo_4$ the $3$-step nilpotent Lie algebra in dimension $4$, with basis $f_1, f_2, f_3, f_4$ and Lie bracket $[f_1,f_2]=f_3$ and $[f_1,f_3] = f_4$. It can be shown that a derivation of $\ngo_4$ in this basis has the following form $\begin{pmatrix}
	a&0&0&0\\
	b&e&0&0\\
	c&f&a+e&0\\
	d&g&f&2a+e
\end{pmatrix}$.
Using \cite{DerOri21} we know that there exist a simply transitive NIL-affine action from $G$ on $H$, moreover we can construct an embedding from $\ggo$ to $\hgo$ by
	$$(x_1,x_2,x_3,x_4)\to\left( (x_3,x_2,x_4,x_1),\begin{pmatrix}
	-\frac{x_1}2&0&0&0\\
	0&x_1&0&0\\
	0&-x_3&\frac{x_1}2&0\\
	0&0&-x_3&0
\end{pmatrix}\right)$$
which induces such a simply transitive NIL-affine action from $G$ on $H$.
Then, the left multiplications of the PLAS can be computed using \eqref{eq: left multiplication} 
$$L_x=\begin{pmatrix}
	-\frac12x_4&0&0&0\\
	0&x_4&0&0\\
	0&-x_1&\frac12x_4&0\\
	0&0&-x_1&0
\end{pmatrix},$$
and the right multiplications are given by
$$R_y=\begin{pmatrix}
	0&0&0&-\frac12y_1\\
	0&0&0&y_2\\
	-y_2&0&0&\frac12 y_3\\
	-y_3&0&0&0
\end{pmatrix}.$$
Note that $R_y$ are not nilpotent for all $y$. 
Moreover, $\ad_y^\hgo=\begin{pmatrix}
	0&0&0&0\\
	0&0&0&0\\
	-y_2&y_1&0&0\\
	-y_3&0&y_1&0
\end{pmatrix},$ and then we can check that $$R_y-\frac12\ad^\hgo_y = \begin{pmatrix}
0&0&0&-\frac12y_1\\
0&0&0&y_2\\
-\frac12 y_2&-\frac12 y_1&0&\frac12 y_3\\
-\frac12y_3&0&-\frac12y_1&0
\end{pmatrix}$$ are not nilpotent for all $y$, since $\frac12(y_1^2y_2)^{1/3}$ is an eigenvalue. Therefore the definitions of complete PLAS on a pair $(\ggo,\hgo)$ will depend on the nilpotency class of $\hgo$.
\end{example}

%\section{Open questions}
%\color{blue}For all nilpotent examples we have that $R_x$ are nilpotent. Is that true in general?
%\color{black} 
%\

\bibliographystyle{plain}
\bibliography{ref}

\end{document}